\documentclass[a4paper,10pt]{amsart}

\usepackage[english]{babel}
\usepackage{mathrsfs}
\usepackage{amsmath}
\usepackage{amsthm}
\usepackage{amssymb}

\theoremstyle{plain}
\newtheorem{thm}{\textbf{Theorem}}[section]

\newtheorem{prop}[thm]{Proposition}

\newtheorem{cor}[thm]{Corollary}

\newtheorem{lem}[thm]{Lemma}

\theoremstyle{definition}

\newtheorem{defi}[thm]{Definition}

\theoremstyle{remark}

\newtheorem{rem}[thm]{Remark}

\newtheorem{ex}[thm]{Example}

\newcommand{\ms}{\mathscr}
\newcommand{\mf}{\mathrm}
\newcommand{\mc}{\mathcal}
\newcommand{\mb}{\mathbf}

\newcommand{\ov}{\overline}

\newcommand{\K}{\mathcal K}
\newcommand{\W}{\mathcal W}
\newcommand{\V}{\mathcal V}
\newcommand{\U}{\mathcal U}

\newcommand{\I}{\mc I}

\renewcommand{\t}{\mathsf{T}}			
\newcommand{\tu}{\mathsf{T_\forall}}			
\newcommand{\tw}{\mathsf{T_W}}

\newcommand{\trap}[1]{\sqrt[\t]{#1}^+}
\newcommand{\trad}[1]{\sqrt[\t]{#1}}

\newcommand{\w}{\mathcal W}
\newcommand{\wk}{\mathcal W_\K}

\newcommand{\uk}{\mathcal U_\K}
\newcommand{\vk}{\mathcal V_\K}

\newcommand{\mo}{\models}
\newcommand{\onto}{\twoheadrightarrow}
\newcommand{\into}{\hookrightarrow}
\renewcommand{\iff}{\Leftrightarrow}
\newcommand{\imp}{\Rightarrow}

\renewcommand{\phi}{\varphi}

\newcommand{\N}{\mathbb N}

\title{Prime types and geometric completeness}
\author{Jean Berthet}

\begin{document}

\begin{abstract}
The geometric form of Hilbert's Nullstellensatz may be understood as a property of ``geometric saturation'' in algebraically closed fields.
 We conceptualise this property in the language of first order logic, following previous approaches and borrowing ideas from classical model theory,
  universal algebra and positive logic. This framework contains a logical equivalent of the algebraic theory of prime and radical ideals, as well
   as the basics of an ``affine algebraic geometry'' in quasivarieties. Hilbert's theorem may then be construed as a model-theoretical property,
    weaker than and equivalent in certain cases to positive model-completeness, and this enables us to geometrically reinterpret model-completeness
     itself. The three notions coincide in the theories of (pure) fields and we apply our results to group-based algebras, which supply a way of dealing with
      certain functional field expansions.
\end{abstract}

\maketitle

 \section*{Introduction}
Several analogues of Hilbert's Nullstellensatz have been proved in various contexts, for example in real algebraic geometry (\cite{BCR}, Theorem 4.1.4)
 or in the study of formally $p$-adic fields (\cite{PR}, Theorem 7.6 and \cite{SK}, Theorem 3). There is also a record of several attempts to
 develop a universal framework unifying these or other results (see for example \cite{CH} III.6, \cite{MK}, \cite{BP}, \cite{WW}). Although they
 exhibit different motivations or points of view, those cited here all use model theory or universal algebra in a broad meaning : the idea
 is that first order logic is suitable for dealing with the different kinds of algebraic structures underlying the Nullstellens\"atze.\\
The work expounded here is extracted and extended from our Ph.D. thesis (\cite{JB}) and finds its place in this general setting.
 We borrow ideas from different approaches to produce a unifying theory in the context of first order logic, building on the
  following considerations. First, G.Cherlin proved that model-completeness readily applies to produce a general Nullstellensatz
   for inductive theories of commutative rings (\cite{CH}, III.6) and McKenna used these ideas to characterise various Nullstellens\"atze in the context
    of model-complete theories of fields (\cite{MK}). Secondly, the stronger notion of positive model-completeness introduced
    by A.MacIntyre has ``geometric'' properties (\cite{MI}) and was reintroduced in another spirit in the language of positive logic
    by I. Ben Yaacov and B. Poizat (\cite{BY} and \cite{BYP}) as an analogue of model-completeness; this logic with its homomorphisms seems to be closer
    to the algebraic basic concepts.
    Thirdly, an inspection of the structure of affine
    coordinate rings in affine algebraic geometry reminded us of the theory of quasivarieties, a piece of universal algebra
     (\cite{HD}, chapter 9). Lastly, the affine objects of algebraic geometry may be defined in the context of equational varieties, a
     special case of quasivarieties (\cite{BP}).\\
Now all this takes place in the realm of first order logic, and it was possible to combine those insights in the following theory we
called ``geometric completeness'' and which consists of a formalisation of the property stated in Hilbert's theorem,
 in connection with the very close notions of model-completeness and positive model-completeness.\\
 After reviewing the basic elements needed
  in model theory, universal algebra and algebraic geometry in part~\ref{BASICS}, we will in part~\ref{FA} build a generalised
  theory of ``logical ideals'' which we call ``$a$-types'' and define the notions of prime and ra\-di\-cal ones as in the ring theory
   underlying algebraic geometry; this culminates in the representation theorem (\ref{rep}), which somehow generalises the representation of
    semiprime rings and makes the algebraic bridge between the quasivarieties of universal algebra and the universal classes of model theoretic algebra.\\
 The $a$-types are then used in part~\ref{GCS} to define a kind of ``formal algebraic geometry'' with ``affine algebraic invariants'' in first order
     structures, quite as in the classical case, after what we introduce through what we call ``geometrically closed structures'', a formal analogue
      of the property expressed in the geometric Nullstellensatz; here formal ideals ($a$-types) as well as first order formulas are used in the expression of
     the property (see Theorem~\ref{CHARGC}); this leans on the two traditions implied, the algebraico-geometric one and the model-theoretic one.\\
 Part~\ref{GC} is the core of the paper, and characterising algebraically closed fields among integral domains and non-trivial rings supplies
 the intuition for connecting the work of section~\ref{GCS} to model theory, in the definition of ``geometrically complete theories''
        and the statement of a ``logical'' Nullstellensatz (Theorem~\ref{POSNS}), the converse of which is tackled under mild assumptions
         in Theorem~\ref{PMCGC}, strongly connecting the theory presented here to positive model-completeness. The remainder of part~\ref{GC}
          explores the connection between geometric completeness, classical model-completeness and Cherlin's original theorem.\\
Part~\ref{APP} deals with some applications, firstly in the particular case of theories of fields, and secondly in the universally
            algebraic setting of group-based algebras, in which the Nullstellensatz and the affine algebraic invariants may be construed in an algebraic
             language, suitable for example for some expansions of fields by additional operators.

 \section{Preliminaries}\label{BASICS} 
We work in and assume standard knowledge of one-sorted first order logic (pre\-di\-ca\-te calculus). The reader is refered to \cite{HD}
or \cite{CHK} for basic model theory. Unless specified, all the concepts and results of this paper are valid in many-sorted first order logic.
Throughout this section, $\ms L$ will denote a
 first order language and as usual, if $\ov x$ is a tuple of variables, $\phi(\ov x)$ will denote an $\ms L$-formula whose free
 variables are \emph{among} $\ov x$ and $|\ov x|$ will stand for the length of $\ov x$. Sometimes we will allow formulas with tuples
 of parameters in a given structure $A$, for which we will use letters of the beginning of the alphabet,
 like in $\phi(\ov a,\ov x)$. We will loosely write $\ms L\sqcup A\sqcup \ov x$ to denote the addition to $\ms L$ of new \emph{constant symbols} for
 naming the elements of $A$ and $\ov x$. We remind the reader that an \emph{elementary class} $\K$ of $\ms L$-structures is the class of models
 of a first order theory $\t$ in $\ms L$. It is noted here $Mod(\t)$ and is naturally closed under isomorphic copies.\\

  Let $\t$ be a first order theory in $\ms L$ and $\K=Mod(\t)$. As we are algebraically-minded here, we will consider not only
  \emph{embeddings} of $\ms L$-structures, but also and mainly \emph{homomorphisms}.
\begin{defi}
 If $A$ and $B$ are $\ms L$-structures and $f:A\to B$ is a map, then we say that $f$ is a \emph{homomorphism},
 if for every \emph{atomic} sentence $\phi(\ov a)$ of $\ms L$ with parameters in $A$ such that $A\mo\phi$,
 one has $B\mo\phi^f$ ($\phi^f$ denotes $\phi(f\ov a)$).\\
Remember that $f$ is an \emph{embedding}, if in addition the converse is true : for every atomic sentence $\phi(\ov a)$ defined in $A$,
 $A\mo\phi$ if and only if $B\mo\phi^f$. If moreover the property is true for \emph{every} sentence $\phi(\ov a)$, then $f$
 is said to be \emph{elementary}.
\end{defi}
   Remember that if $A$ is an $\ms L$-structure, then $D^+A$ denotes the \emph{atomic diagram of $A$}, the set of
  all atomic sentences $\phi$ with parameters in $A$ such that $A\mo\phi$. If $B$ is another $\ms L$-structure, then
   homomorphisms from $A$ into $B$ are in bijection with $\ms L(A)$-expansions $B^*$ of $B$ such that $B^*\mo D^+A$.
    If $f:A\to B$ is a homomorphism and $\phi(\ov a)$ is a sentence with parameters in $A$, we will write $f\mo\phi$
  to mean that $B\mo \phi^f$; in other words the $\ms L(A)$-expansion $(B,f)$ induced by $f$ satisfies $\phi$.
   
 \subsection{Quasivarieties}
This is the occasion to introduce \emph{quasivarieties}, which are elementary classes in which one may carry on
 the basic constructions of algebra, even in a relational language. When we speak about \emph{products} of $\ms L$-structures,
 we allow the \emph{empty} product, which is defined as a \emph{trivial} structure, a structure with one element and such that
 every relational interpretation is full. Such a structure is noted here $\mb 1$.

\begin{defi}\label{QVAR}
$\K$ is a \emph{quasivariety} if it is closed under (any) products and substructures (in particular, if $\mb 1$ is in $\K$).
\end{defi}

Quasivarieties will be used here in connection with a universal version of coordinate algebras of affine algebraic geometry and of
 the representation of semiprime rings (rings with no nilpotent elements). The axiomatisation of quasivarieties is connected to
 certain formulas in the syntax of which are reflected their algebraic properties.

\begin{defi}
An $\ms L$-formula $\phi(\ov x)$ of the form  $\forall \ov y\ \bigwedge\Phi(\ov x,\ov y) \imp \psi(\ov x,\ov y)$, where $\Phi\cup\{\psi\}$ is
 a finite set of atomic formulas, will be said \emph{quasi-algebraic}. The set of quasi-algebraic consequences of $\t$ will be denoted $\tw$.
\end{defi}
Quasi-algebraic formulas are usually called \emph{(basic) universal strict Horn formulas} (see~\cite{HD}, section~9.1). We introduce this
 terminology in order to avoid confusion with Horn formulas in general, and to follow the idea of ''quasi-identities`` introduced by
 B.Plotkin in~\cite{BP} (lecture 3, section 2). Among quasivarieties containing $\K$, we find the class of all $\ms L$-structures,
  and a smallest one, ''generated by $\K$``.

\begin{thm}(\cite{HD}, 9.2.3)\label{GENQV}
The class $Mod(\tw)$, noted here $\wk$, is the smallest quasivariety containing $\K$. It is the smallest class of
 structures containing $\K$ and closed under isomorphic copies, substructures and products.
\end{thm}

\begin{ex}\label{EX1}
The theory of semiprime rings (without nilpotent elements) is quasi-algebraic in $\langle +,-,\times, 0,1\rangle$.
 The class $\W$ of its models is the quasivariety generated by (algebraically closed) fields, or even integral domains.\\
Analogously, call a ring \emph{real} if it satisfies the axioms $\forall x_1,\ldots,x_n\ [\bigwedge_{i=1}^n (x_i^2=0) \imp (x_1=0)]$ and
 those for semisimple rings in the same language. The class of real rings is then the quasivariety generated by the class of
  (subrings of) real (closed) fields  (see \cite{BCR}, section~1.1).\\
Partially ordered sets in the language $\langle \leq\rangle$ form a quasivariety, but not in $\langle < \rangle$.\\
  \end{ex}

\subsection{Universal classes}
Quasivarieties are special cases of \emph{universal (elementary) classes}.
\begin{defi}
 $\K$ is a \emph{universal} class if it is closed under substructures.
\end{defi}

As for quasivarieties, universal classes are axiomatised by their consequences of a certain form.
\begin{defi}
 An $\ms L$-formula is \emph{(basic) universal} if it is of the form \newline $\forall\ov y\ [\bigwedge\Phi(\ov x,\ov y)
\imp\bigvee\Psi(\ov x,\ov y)]$, where $\Phi$ and $\Psi$ are finite sets of atomic formulas. The set of (basic) universal consequences
 of $\t$ is noted $\tu$.
\end{defi}

The following characterisation is well-known in model theory and analogous to Theorem~\ref{GENQV}.
\begin{thm}
The class $Mod(\tu)$, noted here $\uk$, is the smallest universal class containing $\K$. If $A$ is an $\ms L$-strucure, then 
$A\in\uk$ if and only if $A$ embeds into a model of $\t$.
\end{thm}

\begin{rem}
It should be obvious that the quasivariety $\wk$ generated by $\K$ is also generated by $\uk$. Formally, we have $\wk=\W_{\uk}$.
\end{rem}
  
Section~\ref{PRT} will expound a structural relationship between universal classes and quasivarieties, and we will characterise $\wk$
 in Corollary~\ref{REP} by the following operation, under which quasivarieties are also closed.

\begin{defi}
 If $A$ is a structure, $A$ is a \emph{subdirect product} of objects in $\K$, if it is
 isomorphic to a substructure $B$ of a product of objets of $\K$, such that every projection of $B$ onto the components of the
 product is surjective.
\end{defi}
Remark that any subdirect product of structures of a universal class $\U$, is in $\W_\U$.

\begin{ex}\label{EX2}
 Going back to examples~\ref{EX1}, integral domains make up the universal class generated by fields, and real integral domains
  (real and integral) the universal class generated by real fields. The sentence saying that an order is total
   is universal, and lattice-ordered subrings of real closed fields are the totally ordered integral domains.\\
 Every semiprime ring is representable as a subdirect product of integral domains; this is true for real rings (see~\ref{EX1})
  and for $f$-rings (see example~\ref{EX5}), and will be systematised in Theorem~\ref{rep}.
\end{ex}

Analogous to universal formulas are $h$-universal ones, which characterise which structures are the domain of a homomorphism into a
model of $\t$ (see~\cite{BYP} or~\cite{BY}).
 \begin{defi}
  A formula $\phi(\ov x)$ is $h$-universal if it is of the form $\forall\ov y\ \neg\bigwedge\Phi(\ov x,\ov y)$, where $\Phi$
   is a finite set of atomic formulas. The set of $h$-universal consequences of $\t$ is $\t_u$.
 \end{defi}
A syntactical characterisation of homomorphisms into a model of $\t$ is the following.
\begin{prop}
 An $\ms L$-structure $A$ is a model of $\t_u$ if and only if there exists a homomorphism from $A$ into a model of $\t$.
\end{prop}

\subsection{Presentations and functional existentiation}\label{CAR}
Algebraic presentations, as in groups and rings, generalise to quasivarieties.
 \begin{defi}
 An \emph{$\ms L$-presentation} is a couple $(X,P)$, where $X$ is a set of additional constants and $P$ is a set of atomic
  sentences in $\ms L\sqcup X$. A \emph{model} of the presentation $(X,P)$ \emph{in $\K$} is an $\ms L(X)$-structure $A$ which
  satisfies $P$ and whose $\ms L$-reduct $A_0$ is in $\K$.
\end{defi}

\begin{prop}(\cite{HD}, 9.2.2)
 If $\t$ is quasi-algebraic, then every presentation $(X,P)$ has an initial model in $\K$, generated by the interpretation of $X$.
 Such a structure is \emph{presented by $(X,P)$ in $\K$}.
\end{prop}

If $\K$ is a quasivariety, $A$ is a structure in $\K$ and $X$ is a set, the presentation $(X,\emptyset)$ has an initial
 model $A[X]$ in $\K$, which is the analog of a polynomial algebra. In fact, the quotients introduced in section~\ref{FA} may be used
  to explicitly construct those $A[X]$.
Quasi-algebraic theories are special cases of \emph{limit theories}, which employ a little bit of existential quantification
 (\cite{CO2}, section 2). We will use here the underlying idea of existentiation on functional variables in order to define the
 category of affine algebraic varieties in~\ref{AFFQU}.
 
 \begin{defi}
  If $\phi(\ov x,\ov y)$ is a conjunction of atomic formulas such that $\t\mo\forall\ov x,\ov y,\ov y'\ [\phi(\ov x,\ov y)\wedge\phi
  (\ov x,\ov y')\imp \ov y=\ov y']$, then $\phi$ will be called \emph{functional in $\ov x$ (for $\t$)}.
 \end{defi}
 We will be interested in primitive positive formulas where existential quantification is allowed only on functional variables, in order
  to define morphisms of affine varieties in section~\ref{AFFQU}. The following lemma is an easy property of quasivarieties and a version of
  \cite{CO}, II 4.5.
    \begin{lem}\label{QTERMS}
Let $\phi(\ov x)=\exists\ov y\ \psi(\ov x,\ov y)$, where $\psi$ is a conjunction of atomic formulas, functional in $\ov y$ for $\t$.
  If $\t$ is quasi-algebraic and $\t\mo\forall\ov x\ [\phi(\ov x)\iff \theta(\ov x)]$ for a conjunction $\theta$ of atomic formulas,
 then there exist $|\ov y|$ terms $t_j(\ov x)$ such that $\t\mo\forall \ov x\ [\theta(\ov x)\iff\psi(\ov x,\ov t(\ov x))]$.
 \end{lem}
\begin{proof}
 Let $\ov x=x_1,\ldots,x_n$ and $A$ be an initial model of $(\{x_1,\ldots,x_n\},\theta(\ov x))$ in $\K=Mod(\t)$. By hypothesis,
  there are elements $\ov a$ in $A$ such that $A\mo\psi(\ov x^{A},\ov a)$. As $A$ is generated by the interpretations
   of $\ov x$, this means there are terms $\ov t(\ov x)$ such that $\ov a=\ov t^{A}(\ov x^A)$.
 Now let $B\mo\t$ : if $B\mo\theta(\ov b)$, there is a unique morphism $f:A\to B$, $\ov x\mapsto \ov b$,
     so $B\mo\psi(\ov b,\ov t^B(\ov b))$, as $\ov t^B(\ov b)=f(\ov t^A(\ov x^A))$. If $B\mo\psi(\ov b,\ov t(\ov b))$, then obviously
     $B\mo\theta(\ov b)$.
\end{proof}

\subsection{Varieties and group-based algebras}\label{VAR}
 
 \paragraph{Equational varieties}
 Among quasivarieties we find \emph{varieties} which theorise the notion of ``algebraic structure''. We suppose here
 that $\ms L$ is an equational language, i.e. that it contains no relation symbol. Let $\K$ be a class of $\ms L$-structures.
\begin{defi}
A formula is \emph{algebraic} if it is of the form $\forall\ov x\ \phi(\ov x)$, where $\phi$ is atomic. The class $\K$
is a \emph{(equational) variety (in $\ms L$)} if it as axiomatised by a set of algebraic sentences.
\end{defi}

Varieties are characterised by a theorem of G.Birkhoff (\cite{HD}, 9.2.8).
\begin{thm}
Equivalently, $\K$ is a \emph{(equational) variety} if and only if it is closed under (any) products, substructures and
 homomorphic images.
\end{thm}

Notice that varieties are quasivarieties of a certain kind in an equational language. In general, the set $\mathsf{T_V}$ of algebraic
 consequences of $\t$ axiomatises the smallest variety $\vk$ containing $\K$, and we have $\K\subset \wk\subset \vk$. One interest of varieties lies in the description of the homomorphisms by their ``kernels'', which may be defined in a familiar
 way in the following situation.
 
 \paragraph{Group-based algebras}
 
 It is possible to generalise the notion of commutative ring, by considering certain varieties of groups with additional operations,
  introduced in~\cite{HI} (see also~\cite{BP}, lecture 4, section 1).
  \begin{defi}
   Suppose that $\ms L$ has only one sort and contains symbols $*,^{-1},e$ for group operations. An $\ms L$-structure $A$ is a
   \emph{group-based algebra} (for $\ms L$), if its reduct to $\langle *,^{-1},e\rangle$
     is a group (not necessarily abelian) and for every functional symbol $F$ of positive arity $n$,
      we have $F(e,\ldots,e)=e$. $\K$ is a variety of \emph{group-based algebras} in $\ms L$ if every object of $\K$
     is a group-based algebra.
  \end{defi}
Notice that the class of \emph{all} group-based algebras for $\ms L$ is already a variety (axiomatised in $\ms L$
 by the sentences $F(e,\ldots,e)=e$ for all functional symbols $F$), so a variety of group-based algebras
is merely a subvariety of this one. Note also that there is no condition on the \emph{constants} of the language.
\begin{ex}\label{EX3}
 The class of rings in the language of rings is a variety a group-based algebras, because $0\times 0=0$. This is also
  the case for Lie rings, which are abelian groups equipped with a Lie Bracket (see~\cite{BP} for this and other examples).\\
 Closer to model theory are the variety of differential rings in $\langle +,-,\times,d,0,1 \rangle$ (Example~\ref{APEX1})
  and the variety of rings with an additional endomorphism in $\langle +,-,\times,\sigma,0,1\rangle$ (Example~\ref{APEX2}).
\end{ex}
A homomorphism $f:A\to B$ of group-based algebras is in particular a group homomorphism, so the \emph{algebraic}
kernel of $f$ is at least a normal subgroup of $A$.
\begin{defi}
 An \emph{ideal} of a group-based algebra $A$ in $\ms L$ is a normal subgroup $I$ of $A$, such that for every functional
  symbol $F$ of $\ms L$, of arity $n$, and every $n$-tuples $\ov a\in I^n$, $\ov b\in A^n$, one has $F(\ov a)\in I$ and
 $F(\ov a)^{-1}*F(\ov b)^{-1}*F(\ov a*\ov b) \in I$. The element $F(\ov a)^{-1}*F(\ov b)^{-1}*F(\ov a*\ov b)$ is the
 \emph{commutator} of $\ov a$, $\ov b$ and $F$ ($\ov a*\ov b$ denotes the coordinatewise product of the tuples).
\end{defi}

  \subsection{Existentially closed structures and model completeness}
The notion of model-completeness may be introduced along with the broader notion of existential completeness, which in turn may be
  expressed using universal formulas. Replacing these by quasi-algebraic ones, we will get an analog of the ``geometric saturation''
  we find in algebraically closed fields. Positive existential completeness is as well an analog piece of first order logic
   (as in~\cite{MI}), or may be seen as a generalisation of first order logic in the context of ``positive model theory''
   (as in~\cite{BY} and \cite{BYP}). For the remainder of this section, $\K$ is \emph{any} class of $\ms L$-structures.
\begin{defi}
 An $\ms L$-formula is \emph{existential} if it is a prenex formula mentionning only existential quantifiers in its prefix. An
  existential formula is \emph{positive (existential)}, or \emph{coherent}, if it does not mention any negation.
\end{defi}

Existential and coherent sentences with parameters are ``preserved'' by respectively, embeddings and homomorphisms. The theory
 of existential completeness deals with the reverse property.
\begin{defi}
An embedding of $\ms L$-structures $f:A\to B$ is \emph{existentially closed} if for every existential sentence
$\phi(\ov a)$ with parameters in $A$ such that $B\mo \phi(f\ov a)$, one has $A\mo\phi(\ov a)$. Likewise, a homomorphism
$f:A\to B$ is an \emph{immersion} if it reflects the validity of every coherent sentence $\phi(\ov a)$ with parameters in $A$.
\end{defi}
Notice that existentially closed embeddings are immersions, because immersions are embeddings and coherent formulas are existential.
 The model-theoretic following notions of ``saturation'' lie at the heart of the methods and problems expounded here. 
\begin{defi}
 If $A$ is an object of $\K$, then $A$ is said \emph{existentially closed in $\K$} if every \emph{embedding} $f:A\into B$ in $\K$
  is existentially closed. It is \emph{positively existentially closed in $\K$}, if every \emph{homomorphism} $f:A\to B$ in $\K$ is an immersion.
\end{defi}
 Notice that every existential formula is equivalent to a finite disjunction of ``primitive formulas'', i.e. of the form
  $\chi(\ov x)=\exists \ov y\ [\bigwedge_{i=1}^n\phi_i(\ov x,\ov y)\wedge \bigwedge_{j=1}^m\neg\psi_j(\ov x,\ov y)]$, where the
   $\phi_i$'s and the $\psi_j$'s are atomic. Likewise, every coherent formula is equivalent to a finite disjunction of
    positive primitive ones (the same form as $\chi$ but without the negations). This means existential completeness
     may be reformulated as follows.
\begin{prop}
Let $f:A\to B$ a map of $\ms L$-structures. If $f$ is an embedding, then it is existentially closed if and only if, for every (basic) universal sentence
   $\phi(\ov a)$ defined in $A$ and such that $A\mo\phi(\ov a)$, one has $B\mo\phi(f\ov a)$. If $f$ is a homomorphism, then it is an
   immersion if and only if it preserves the validity of $h$-universal sentences $\phi(\ov a)$ defined in $A$ (i.e. of the form
   $\forall \ov y\ [\neg\bigwedge\Phi(\ov a,\ov y)]$, where $\Phi$ is a finite set of atomic formulas).
\end{prop}
 This is the spirit in which we will introduce \emph{geometric completeness} in sections~\ref{GCS} and~\ref{GC}.
  We will focus here on the situations where $\K$ is elementary, and sometimes where the subclass of (positively) existentially closed
structures in $\K$ also is.
\begin{defi}
 Suppose $\K=Mod(\t)$ is elementary. If every object of $\K$ is existentially closed in $\K$, then $\t$ (and then $\K$) is said
 \emph{model-complete}. If every object of $\K$ is positively existentially closed in $\K$, then $\t$ (and then $\K$) is said
 \emph{positively model-complete}.
\end{defi}

 We mention that the analogy between existential and positive existential completeness goes beyond what is presented here, as the first
 notion may be reduced to the second through ``Positive Morleyisation'' (see~\cite{BY}, 1.2 and \cite{BYP}). 
   Moreover, every positively model-complete theory is model-complete, which was in the definition given in~\cite{MI}, section 2, but
   here is a consequence of Lemma~\ref{POSEL}, itself deriving from the following one.

\begin{lem}(\cite{BYP}, lemma 15)\label{PMCCOH}
 If $\t$ is positively model-complete, then every coherent formula has a coherent complement modulo $\t$.
\end{lem}

   \subsection{Nullstellens\"atze}\label{HCNS}
Speaking about ``saturated'' structures, in algebraic geometry the fundamental ones are algebraically closed fields. One way
 of stating Hilbert's Nullstellensatz is the following ``geometric'' form.
 
\begin{thm}(\cite{RH}, 1.2)\label{HILNUL}
 Let $k$ be an algebraically closed field and $I$ be an ideal of a polynomial algebra $k[X_1,\ldots,X_n]$ over $k$.
  The set $\ms I(\ms Z(I))$ of all polynomials of $k[\ov X]$ vanishing on the set $\ms Z(I)$ of zeros of $I$ in
   $k^n$, is the algebraic radical of $I$, $\sqrt{I}$.
\end{thm}

In the ``language of rings'' $\langle +,-,\times,0,1\rangle$, the theory of algebraically closed fields is model-complete (\cite{MA}),
which is an easy corollary to the Nullstellensatz, and inspired G.Cherlin to develop the following little theory.

\begin{defi}
 Let $\t$ be a theory of commutative rings and $A$ be a ring. Let $X$ be any set and $I$ be an ideal of $A[X]$.
The \emph{$\t$-radical of $I$}, noted $\trad I$, is the intersection of all ideals $J$ of $A[X]$ such that $I\subseteq J$,
$A[X]/J\mo\tu$ and $J\cap i(A) =(0)$, where $i(A)$ denotes the copy of $A$ in $A[X]$.
\end{defi}
In other words, we keep for the $\t$-radical those ideals which are the kernel of an evaluation morphism at a rational
point of $I$ in an extension of $A$, model of $\tu$ (or $\t$). 

\begin{thm}\label{CH}(\cite{CH}, III.6, theorem 73)
 Let $\t$ be a theory of commutative rings and $A$ an existentially closed model of $\t$. If $X=X_1,\ldots,X_n$ is a finite
  tuple of variables, $\I$ a finitely generated ideal of $A[X]$ and $P$ a polynomial of $A[X]$, then $\ms Z_A(I)\subset\ms Z_A(P)$
   if and only if $P\in\trad I$.
\end{thm}

\section{A formal algebra}\label{FA}
Let $A$ be an $\ms L$-structure. If $f:A\to B$ is a homomorphism, then the set of all atomic sentences with parameters
 in $A$ satisfied in $B$ through $f$, characterises the image of $f$, much as an ideal in ring theory. In this section, we develop
 a little theory of ``abstract ideals'' which we call $a$-types. For this purpose, as we focus on homomorphisms,
 we introduce the notation $f\models \phi$, if $\phi(\ov a)$ is a sentence with parameters in $A$ : this means that
 $B\models \phi(f\ov a)$.

  \subsection{Algebras and $a$-types}
\begin{defi}
An $\ms L$-homomorphism $f:A\to B$ will be called an \emph{$A$-algebra}.\\
If $f:A\to B$ and $g:A\to C$ are $A$-algebras, then a \emph{morphism of $A$-algebras} from $f$ into $g$ is
 an $\ms L$-homomorphism $h:B\to C$, such that $h\circ f=g$.
\end{defi}

We will factorise algebras based on the structure $A$ by their ``logical kernel'', which will be defined as a particular case of
 the following ``linguistic'' notion.
\begin{defi}
 An \emph{$a$-type of $A$} ($a$ for ``atomic'') is a set of atomic sentences of $\ms L(A)$.
 If $f:A\to B$ is an $\ms L$-homomorphism, then the \emph{$a$-type of $f$}, noted $tp_a(f)$, is the set of all atomic
 sentences with parameters in $A$, satisfied by $f$.
\end{defi}

If $\pi$ is an $a$-type of $A$, we want to define the ``quotient of $A$ by $\pi$'', so that this quotient
 be a ``universal'' model of $\pi$. Let then $\tilde\pi$ be the set of $L(A)$-atomic sentences $\phi$ such that
  $D^+A\cup\pi\models\phi$. We
 define on the elements of $A$ an equivalence relation $\sim$ by putting $a\sim b$ if and only if the formula $a=b$ is in $\tilde\pi$.
\begin{defi}
 An $a$-type $\pi$ of $A$ is \emph{closed} if $\pi=\tilde\pi$.
\end{defi}
 Then define as in \cite{BY}, 1.25, an $\ms L$-structure on the quotient $A/\pi$, interpreting each sort symbol in a natural way : the
  set of equivalence classes of elements of its interpretation in $A$. If $F$ is a functional symbol of $\ms L$ and $\ov a$ is an appropriate tuple of $A$, then
 $F^\pi([\ov a]):=[F(\ov a)]$ describes the interpretation of $F$ in $A/\pi$, and if $R$ is a relational symbol, we define its extension
 $R^\pi$ in $A/\pi$ as the set of appropriate tuples $[\ov a]$ such that the sentence $R(a)$ is in $\tilde\pi$ (here $[\ov a]$ denotes
 the equivalence class of the tuple $\ov a$). We leave to the reader to check that this is well-defined. The canonical projection is then a
 homomorphism $f_\pi:A\to A/\pi$ because is satisfies $D^+A$, and by construction $f_\pi\models \pi$.

\begin{defi}
 The homomorphism $f_\pi:A\to A/\pi$ is the \emph{quotient of $A$ by $\pi$}.
\end{defi}

If $f:A\to B$ is an homomorphism and $f\models \pi$, then put $g([a]):=f(a)$ for all $a\in A$. If $a,b\in A$ and $a\sim b$, then
 $f\models a=b$, so we readily see that $g$ well defines an homomorphism of $A$-algebras from $f_\pi$ to $f$, which is necessarily
 unique by construction; this is summarised in the following

\begin{prop}
 If $f:A\to B$ is a homomorphism and $f\mo\pi$, then there exists a unique homomorphism of $A$-algebras from $f_\pi$
 to $f$.
\end{prop}
This universal property easily leads to the following ``isomorphism theorem''.

\begin{thm}\label{iso}
 Let $f:A\to B$ be a homomorphism and $\pi=tp_a(f)$. We have an isomorphism $A/\pi\simeq f(A)$.
\end{thm}

  \subsection{Prime and radical $a$-types}\label{PRT}
Here we introduce in this syntactic context the notions of \emph{prime} and \emph{radical ideals} of a commutative ring with unit.
The underlying idea of this development is the observation that a prime ideal in a ring is the kernel of a homomorphism into
 an integral domain, and that an integral domain is always a \emph{subring} of an algebraically closed field. Furthermore,
 any radical ideal is the intersection of the primes containing it, which is interpreted as a characterisation of
 semisimple rings in the \emph{representation theorem} for these rings (see \cite{CH}, III.5, definition 66). We will generalise these
 definitions and this representation, and connect them later to a generalised Nullstellensatz through the theory of quasivarieties : indeed, the affine
 coordinate ring of an affine algebraic variety is a semisimple ring, ``represented'' as a subring of a product of
  integral domains.\\
Let then $\t$ be a first order theory in $\ms L$, $\K$ the class of its models and $A$ an $\ms L$-structure.

 \begin{defi}\label{PRAT}
  Let $\pi$ be an $a$-type of $A$. 
  We will say that : 
\begin{itemize}
 \item $\pi$ is \emph{($\t$-)prime} if it is closed and the quotient structure $A/\pi$ is a model of $\tu$
 \item the intersection of the set of all prime $a$-types $\mf p$ containing $\pi$ is the \emph{(positive $\t$-)radical of $\pi$},
 written $\trap \pi$
 \item $\pi$ is \emph{(positively $\t$-)radical} if $\pi=\trap\pi$.
\end{itemize}
 \end{defi}

 \begin{rem}
 If $\pi$ is an $a$-type of $A$, then a compactness argument shows that $\pi$ is prime if and only if for every finite set $\Phi$ of atomic
 sentences with parameters
  in $A$, if $\t\cup D^+A\cup\pi\mo\bigvee \Phi$, then $\Phi\cap \pi\neq\emptyset$, i.e. there is $\phi$ in $\Phi$ such that
  $\phi\in \pi$; this reminds us of the definition of a prime ideal in ring theory.\\
  Similarly, if $\pi$ is radical, then $\pi$ is closed and in fact, the next theorem (\ref{rep}) may be used to show that if
  $\phi(\ov a)$ is a sentence with parameters in $A$ such that $\t\cup\pi\mo\phi$, then $\phi\in \pi$ (we may replace $\t$ by $\tw$;
  by compactness, this means that there is a finite $\pi_0(\ov a,\ov b)\subset\pi$, such that
  $\t\mo\forall\ov x\ov y\ [\bigwedge\pi_0(\ov x,\ov y)\imp\phi(\ov x)]$). This may alternatively be seen in the fact that the $\t$-radical
  $a$-types are the $\tw$-prime ones (see the comment following Theorem~\ref{rep}).
 \end{rem}

 \begin{ex}\label{EX4}
We go back to examples~\ref{EX1}.\\
If $\t$ is the theory of algebraically closed fields, then the prime $a$-types of a ring $A$ are in bijection with prime ideals,
 whereas radical $a$-types correspond to radical ideals.\\
If $\t$ is the theory of real-closed fields, then the prime $a$-types correspond to real prime ideals and radical $a$-types
 to real ideals (see~\cite{BCR}, section 4.1 for these notions).\\
If $\t$ is the theory of real-closed ordered fields in the language $\ms L=\langle +,-,\times,<,0,1\rangle$ of strictly ordered rings,
 every ring can be turned into an $\ms L$-structure with the empty interpretation for $<$. If then $P$ is a prime cone of $A$
  (see~\cite{BCR}, 4.2 and 4.3), one may associate to $P$ the $a$-type of the homomorphism $A\to (A/supp(P),<_P)$, where $<_P$ is the
   induced total order on $A/supp(P)$. In fact, this establishes a bijection with the prime $a$-types of $A$.\\
If $\t$ is the ``same'' theory in the language with large order $\langle +,-,\times,\leq,0,1\rangle$, turn a ring $A$
 into an $\ms L$-structure by interpreting $\leq$ as equality. In the same way, prime cones are identified with prime $a$-types in $A$.
 \end{ex}

Notice that different theories may give rise to the same notion of prime and/or radical $a$-types. For instance, the theories
 of integral domains, of fields and of algebraically closed fields in the language of rings, produce the same notion
 of prime $a$-types, which correspond to prime ideals. Analogous is the case of real fields and real-closed fields and
  their subrings. This obviously comes from the fact that the notion of prime $a$-type
 is based on the \emph{images} of homomorphisms. The puzzling fact will be that this suffices to develop an analog
 of Hilbert's Nullstellensatz, building on the deeper observation that in some way, the radical $a$-types are
 enough for the theory (sections~\ref{GCS} and~\ref{GC}). Here we show how to interpret the definition in terms of the quasivariety
 generated by $\K$. Indeed, if $\ms P$ is the set of all prime $a$-types of $A$, then for every $\mf p\in \ms P$ we have
  the canonical projection $f_\mf p:A\onto A/\mf p$.

\begin{defi}
 The \emph{representation of $A$ relatively to $\t$}, is the product $f_{\ms P}:A\to \prod_{\mf p\in \ms P} A/\mf p$ of the quotient morphisms
  $f_\mf p:A\onto A/\mf p$.
\end{defi}
If $\pi$ is an $a$-type of $A$ and $f_\pi:A\onto A/\pi$ is the projection, if $f$ denotes the representation of $A/\pi$ relatively
 to $\t$, then the radical $\trap\pi$ of $\pi$ is nothing else than the $a$-type of the composite morphism $f\circ f_\pi$. This gives
 us an ``algebraic'' characterisation of radical $a$-types (Theorem~\ref{rep}).
 
\begin{lem}\label{REPQV}
 The quasivariety $\wk$ generated by $\K$ is the class of all structures which are isomorphic to a substructure of a product of
  models of $\t$.
\end{lem}
\begin{proof}
Call ``admissible'' a structure isomorphic to a substructure of a product of objects of $\K$. Such a structure is certainly in $\wk$.
Suppose now that $(B_i)_{i\in I}$ is a family of admissible structures : every $B_i$ may be construed as a substructure of $C_i$,
itself embedded in a product $\prod_{j\in J_i}C_{j}$ of objects of $\K$. This means that we have an embedding of
 $B=\prod_{i\in I} B_i$ into $\prod_{i\in I} (\prod_{j\in J_i} C_j)$, which shows that $B$ itself is admissible. As
  any substructure of an admissible one is admissible, all this shows that every structure obtained from $\K$ by
   a finite iteration of products or substructures is admissible. By Theorem~\ref{GENQV}, this means that $\wk$ is exactly
    the class of admissible structures.
\end{proof}

\begin{thm}\label{rep}
 An $a$-type $\pi$ is radical if and only if $A/\pi$ is a model of $\tw$.
\end{thm}
\begin{proof}
Write as before $f$ for the representation of $A/\pi$ and $g$ for the composition $f\circ f_\pi:A\onto A/\pi \to B$, where $B$ is
 the product of the prime quotients of $A/\pi$.
 Suppose that $\pi=\trap\pi$. As $\trap\pi=tp_a(g)$ this means by the isomorphism theorem (\ref{iso}) that $f$ is an embedding
and $A/\pi$ is in $\w_{\uk}=\wk$.\\
For the converse, suppose that $A/\pi$ is a model of $\tw$. By the characterisation of $\wk$ from the preceding Lemma~\ref{REPQV},
there exists a (possibly empty) family $(M_i)_{i\in I}$ of models of $\t$, and an embedding $g:A/\pi\into \prod_{i\in I} M_i$.
If $g_i:\prod_I M_i\onto M_i$ denotes the $i^{th}$ projection of the product, let $\mf p_i:=tp_a(g_i\circ g\circ f_\pi)$ : $\mf p_i$ is a prime $a$-type
 containing $\pi$. Now as $g$ is an embedding, $\pi$ is the intersection of the $\mf p_i$'s, by definition of the satisfaction
 of atomic sentences in products : we have $\trap\pi\subset\bigcap_{i\in I} \mf p_i =\pi\subset\trap\pi$, so $\pi=\trap\pi$
 is radical.
\end{proof}

Remark that, as quasivarieties are universal classes, by this theorem a $\t$-radical $a$-type is the same
 thing as a $\tw$-prime one.

\begin{cor}\label{REP}
 The quasivariety $\wk=Mod(\tw)$ generated by $\K=Mod(\t)$ is the class of subdirect products of objects of $\uk=Mod(\tu)$.
\end{cor}
\begin{proof}
 A subdirect product of objects of $\uk$ is in $\w_{\uk}=\wk$, so we only need to prove that an object $A$ of $\wk$ is such
 a subdirect product. By the theorem, $\trap{D^+A}=D^+A$, because $A/D^+A\simeq A$, so the representation $g:A\to \prod_{\ms P} A/\mf p$
 of $A$ relatively to $\t$ is an embedding. As $g_\mf p\circ g$ is surjective for every $\mf p$
 ($g_\mf p$ is the $\mf p^{th}$ projection of the product), $A$ is a subdirect product of the $A/\mf p$'s, which are in $\uk$.
\end{proof}

\begin{ex}\label{EX5}
The class of semisimple $f$-rings is the quasivariety generated by real-closed ordered fields in
$\langle +,-,\times,\wedge,\vee,0,1\rangle$, the language of lattice-ordered rings (\cite{GAR}, 9.1.1 and 9.3.1).
\end{ex}

The second corollary to the representation theorem gives us an ``explicit'' description of universal models of presentations in $\wk$.

\begin{cor}\label{PRESWK}
 If $(X,P)$ is an $\ms L$-presentation and $A$ is the term algebra of the language $\ms L\sqcup X$, then $A/\trap P$ is presented
  in $\wk$ by $(X,P)$.
\end{cor}
\begin{proof}
 We look at $P$ as an $a$-type of $A$; by the representation theorem, the $\ms L$-reduct of $A/\trap P$ is in $\wk$ and is clearly a model of $P$. Let $f:X\to B$
  be a map into an object $B$ of $\wk$, such that $B^f=(B,fX)$ is a model of $(X,P)$. By universal property of the quotient $A/P$
   in $\ms L\sqcup X$, there exists a unique $\ms L\sqcup X$-homomorphism $\tilde f:A/P\to B$. 
    As $B$ is in $\wk$, the $a$-type of $\tilde f$ is radical and contains $P$, so it contains the smallest such one, i.e. $\trap\pi$.
    This means that $\tilde f$ factors itself uniquely through an $\ms L\sqcup X$-homomorphism $g:A/\trap\pi\to B^f$. This also
     means that $g$ is the unique $\ms L\sqcup X$-homomorphism from $A/\trap\pi$ into $B^f$. By definition, $A/\trap\pi$ is
     presented in $\wk$ by $(X,P)$.
\end{proof}

\section{Geometrically closed structures}\label{GCS}
   \subsection{Affine algebraic geometry in quasivarieties}\label{AFFQU}
    
 In~\cite{BP}, B.Plotkin defines the basic notions of algebraic geometry in varieties of algebras. He there speaks of
 ``quasi-identities'', which he uses as a ``signature'' of the rational points of a variety in a structure (lecture 3, section 2).
 Those are the motivation for our terminology of ``quasi-algebraic formulas'', and we introduce here analog considerations in the more
 general context of quasivarieties, i.e. elementary classes axiomatised by quasi-algebraic sentences.\\
 Let us begin with a few generalities. If $A$ is an $\ms L$-structure and $\ov x$ is a finite tuple of variables, then the
 term algebra in the language $\ms L\sqcup A\sqcup\ov x$ comes with a natural structure of $A$-algebra $A\to A[\ov x]$,
  which induces a natural expansion of $A[\ov x]$ to an $\ms L(A)$-structure.
 
\begin{defi}
 The structure $A[\ov x]$ will be called the \emph{($A$-)algebra of terms with parameters in $A$ and in variables $\ov x$}.
\end{defi}
In the quasivariety of all $\ms L$-structures, these algebras of terms are the ``po\-ly\-nomial algebras'' of section~\ref{CAR} : if $\ov b$ is an
$|\ov x|$-tuple of a structure $B$ and $f:A\to B$ is a homomorphism, there exists a unique homomorphism $\tilde f$ of $A$-algebras from
$A[\ov x]$ to $B$ such that $\tilde f(\ov x)=\ov b$. In particular, for every $|\ov x|$-tuple $\ov a$ in $A$, we have the evaluation
morphism $e_{\ov a}:A[\ov x]\onto A$, which $a$-type is $tp_a^A(\ov a)$, so $A[\ov x]/tp_a^A(\ov a)\simeq A$ by the isomorphism theorem
\ref{iso}.\\
We will develop in section~\ref{GCSS} an analog (in fact, a true generalisation) of algebraically closed fields, from a formalisation of Hilbert's Nullstellensatz. The ideals of rings have been formalised as ``$a$-types'', and a finitely generated
 ideal of a ring of polynomials in finitely many variables is formalised here as a ``finite'' $a$-type $\pi$ of an $A$-algebra of the
 form $A[\ov x]$. Notice that for the canonical expansion of $A[\ov x]$ to $\ms L(A)$, we have $A[\ov x]\mo D^+A$; in fact $D^+A$ axiomatises
 $D^+(A[\ov x])$ in $\ms L\sqcup A\sqcup\ov x$.\\
In this context, $\ms L$-structures will play the role of rings and $a$-types in finitely many variables the role
of sets of polynomials in finitely many variables. The analog of the set of zeros $\ms Z_A(T)$ in a ring $A$ of a
subset $T$ of a polynomial algebra in $n$ variables is the set $\pi(A)=\{\ov a\in A^{|\ov x|}
  : A\mo\bigwedge\pi(\ov a)\}$, where $\pi$ is an $a$-type of $A[\ov x]$. Likewise, the analog of the set of polynomials $\ms I(S)$ vanishing
 on a subset $S$ of a finite power $A^n$ of a ring $A$ is the set $tp_a^A(S)=\bigcap_{\ov a\in S} tp_a^A(\ov a)$, where
 $S\subset A^{|\ov x|}$.\\

Let now $\t$ be a quasi-algebraic theory, $\W=Mod(\t)$ and $A$ be a model of $\t$. We want to define affine algebraic varieties in
 $A$ and their morphisms, much in the same way as what is usually done in affine algebraic geometry. The most natural way of achieving
 this is to replace the finite sets of polynomial equations in rings by finite sets of atomic formulas.

  \begin{defi}
   If $n$ is a natural number, a subset $V$ of $A^n$ is an \emph{affine (algebraic) variety (over $A$)}, if there exists a
 \emph{finite} $a$-type $\pi$ of $A[x_1,\ldots,x_n]$ such that $V=\pi(A)$.
  \end{defi}

The definition of \emph{morphisms} between affine algebraic varieties should then be construed along these lines, i.e.
 morphisms should be defined by finite $a$-types defining a function in $A$ between two varieties, or equivalently by finite
 conjonctions of atomic formulas. Unfortunately, there is a difficulty in composing them if we define them in this way, because
 we should use an existential quantifier to build a formula defining the composite morphism. We would at least need to ``eliminate''
 those existential quantifiers, so we will work in a certain quasi-algebraic theory in order to apply the notions of section~\ref{CAR}.

\begin{defi}
 The \emph{quasi-algebraic diagram} of $A$, noted $Th_W(A|A)$, is the set of all quasi-algebraic sentences \emph{with parameters in $A$}
 satisfied in $A$, for the ``canonical'' expansion of $A$. 
\end{defi}
 The canonical expansion of $A$ is of course the interpretation of the additional symbol constants of $\ms L(A)$ by their
 corresponding elements in $A$. The theory $Th_W(A|A)$ is then an axiomatisation of a quasivariety noted $\W(A)$ in which
 we find (the canonical expansion of) $A$ and other forthcoming $A$-algebras. Notice indeed that $\t\cup D^+A\subset Th_W(A|A)$.

\begin{defi}
 An $A$-algebra $f:A\to B$ is \emph{reduced} if $B$ is in $\W$. We will note $\W_A$ the \emph{category} of reduced $A$-algebras.
\end{defi}
The reader should not merge $\W_A$ with $\W(A)$, but keep in mind that every object in $\W(A)$ has a natural structure
 of reduced $A$-algebra, so that $\W(A)$ embeds in $\W_A$. The equivalence of the two classes will be the subject of the next
 section~\ref{GCSS}.

 \begin{defi}\label{MORVAR}
  If $V=\pi_1(A)$ and $W=\pi_2(A)$ are two affine varieties of $A$, say $\pi_1\subset A[x_1,\ldots,x_n]=A[\ov x]$ and
  $\pi_2\subset A[y_1,\ldots,y_m]=A[\ov y]$, then a \emph{morphism} from $V$ into $W$ is a map $f:V\to W$, for which there exists
   a conjonction $\phi(\ov a,\ov x,\ov y)$ of atomic formulas (with parameters in $A$), with the following properties :
 \begin{itemize}
  \item $\phi$ is functional in $\ov x$ modulo $Th_W(A|A)$
  \item $Th_W(A|A)\mo\forall\ov x\ [\bigwedge\pi_1(\ov x)\iff \exists\ov y\ \phi(\ov x,\ov y)]$
  \item $Th_W(A|A)\cup \{\phi(\ov x,\ov y)\}\mo\pi_2(\ov y)$
  \item for all $(\ov b,\ov c)\in V\times W$, $f(\ov b)=\ov c\iff A\mo\phi(\ov b,\ov c)$.
 \end{itemize}    
 \end{defi}
In other words, we will keep for morphisms between varieties in $A$ only those con\-jonc\-tions of atomic formulas
 which define a functional relation modulo $Th_W(A|A)$, and not merely in $A$. This means that in contrast to the definition of
  affine varieties, the notion of \emph{morphism} is connected to what happens in $A$, and is not merely ``formal''; it even does not
  only depend on $\t$. As all this is valid when $\t=\emptyset$, the hypothesis that $A\mo\t$ is purely a matter of keeping
   in mind that we will deal with structures other than $A$, and we want to choose them in a specified quasivariety.\\

\begin{rem}\label{EQMOR}
 Notice that if two conjunctions of atomic formulas $\phi_1$ and $\phi_2$ with the same free variables define two
  morphisms $f_1$ and $f_2$ and are equivalent modulo $Th_W(A|A)$, then they define the same morphism in $A$.
\end{rem}

If $f:V\to W$ and $g:W\to Z$ are two morphisms, represented by $\phi(\ov x,\ov y)$ and $\psi(\ov y,\ov z)$, it is natural to represent
 the composition $g\circ f$ by the formula $\exists\ov y\ \phi(\ov x,\ov y)\wedge\psi(\ov y,\ov z)$;
 here lemma~\ref{CAR} will help eliminating the existential quantifier. 
  
  \begin{prop}
   The composition $g\circ f$ is a morphism, i.e. it is represented by a conjonction of atomic formulas.
  \end{prop}
\begin{proof}
 By hypothesis, $\phi$ is functional in $\ov x$ modulo $Th_W(A|A)$, so by Lemma~\ref{QTERMS}, there is an appropriate tuple of terms $\ov t(\ov x)$ with parameters in $A$, such that
    $\exists\ov y\ \phi(\ov x,\ov y)$ is equivalent to $\phi(\ov x,\ov t(\ov x))$ in $Th_W(A|A)$. This means that
     $\exists\ov y\ [\phi(\ov x,\ov y)\wedge\psi(\ov y,\ov z)]$ is equivalent to
      $\phi(\ov x,\ov t(\ov x))\wedge\psi(\ov t(\ov x),\ov z)$. This in turn entails that this last formula, functional
       in $\ov x$ modulo $Th_W(A|A)$, represents $g\circ f$.
 \end{proof}

In fact, all this shows that we may restrict ourselves to terms in order to define morphisms, in every such structure $A$.

\begin{lem}\label{REDMOR}
 Let $f:V\to W$ be a morphism. There exists a $|\ov y|$-tuple $\ov t(\ov x)$ of terms with parameters in $A$, such
 that for every $\ov a\in V$, we have $f(\ov a)=\ov t^A(\ov a)$.
\end{lem}
\begin{proof}
 By lemma~\ref{QTERMS}, let $\ov t(\ov x)$ be a witness for $\phi(\ov x,\ov y)$ representing $f$, such that we have
 $Th_W(A|A)\mo\forall\ov x,\ov y\ \phi(\ov x,\ov t(\ov x))\iff \phi(\ov x,\ov y)$. If $\ov a\in V$, then $A\mo\phi(\ov a,f(\ov a))$,
 so by functionality of $\phi$ we have $f(\ov a)=\ov t^A(\ov a)$. 
\end{proof}
Every morphism $f:V\to W$ represented by $\phi(\ov x,\ov y)$ is then also defined by $\bigwedge\pi_1(\ov x)\wedge \bigwedge \pi_2(\ov x)
 \wedge \bigwedge_{j=1}^m y_j=t_j(\ov x)$, for some witnesses $t_j$'s.

The introduction of morphisms enables us to speak of the category $Aff_A$ of affine algebraic varieties over $A$, and we describe now the ``algebraic invariants''
 of affine varieties.
 
  \begin{defi}
 The \emph{affine coordinate algebra} $A[V]$ of $V=\pi_1(A)$ is the quotient of the term algebra
 $A[\ov x]$ by the $a$-type $tp_a^A(\pi_1(A))$.
 \end{defi}
 
Each coordinate algebra of the type $A[V]$ comes with its natural structure of $A$-algebra, which induces a canonical expansion to
$\ms L(A)$. In fact, in the category of $A$-algebras we have the factorisation $A[\ov x]\onto A[V]\into A^V$ induced by the product of
the morphisms $e_{\ov a}:A[\ov x]\to A$ corresponding to each point $\ov a$ of $V$, which also means that $A\to A[V]$ is in $\W_A$ (it is reduced).
 The structure $A^V$ is also equipped with its structure of $A$-algebra induced by the diagonal embedding and as $Th_W(A|A)$ is
 quasi-algebraic, it is stable under products and substructures, so the corresponding $\ms L(A)$-expansion of $A^V$ is in $\W(A)$, as well
 as $A[V]$.
  
  \begin{prop}
   The natural $\ms L(A)$-structure on $A[V]$ is presented in $\W(A)$ by $(\ov x,\pi_1)$.
  \end{prop}
\begin{proof}
 Suppose that $f:A\to B$ is the $A$-algebra structure corresponding to a model $B$ of $Th_W(A|A)$ and that $\ov b$ is an $n$-tuple of
 $B$ such that $B\mo\bigwedge\pi_1(\ov b)$. By the universal property of the term algebra $A[\ov x]$, there is a unique (evaluation)
 morphism $f_{\ov b}:A[\ov x]\to B$, $\ov x\mapsto \ov b$. If $\phi(\ov x)\in tp_a^A(V)$, we have
 $A\mo\forall\ov x\ [\bigwedge\pi_1(\ov x)\imp \phi(\ov x)]$, so this last sentence is in $Th_W(A|A)$, which means that $B\mo\phi(\ov b)$
  by hypothesis on $B$. We conclude that $tp_a^A(V)\subset tp_a^B(\ov b/A)$, where the last denotes the set of atomic $\ms L(A)$-formulas satisfied in
   $B$ by $\ov b$, i.e. the $a$-type of $f_{\ov b}$. Finally, $f_{\ov b}$ factorises uniquely through $A[V]$, which in turn entails that
    $(A[V],\ov x)$ is an initial model of the presentation $(\ov x,\pi_1)$ in $\W(A)$, generated by $\ov x$.
\end{proof}

Let now $f:V\to W$ be a morphism of varieties (with the same notations), represented by a conjonction of atomic $\ms L(A)$-atomic
formulas $\phi(\ov x,\ov y)$, which we may suppose by Lemma~\ref{REDMOR} to be of the form
$\pi_1(\ov x)\wedge\pi_2(\ov y)\wedge\bigwedge_{j=1}^m y_j=t_j(\ov x)$ for some $\ms L(A)$-terms $t_j$'s. By definition of a morphism
 and by the proposition, as the class of $\ov x$ is $A[V]$ satifies $\pi_1$, we also have $A[V]\mo\bigwedge \pi_2(\ov t(\ov x))$.
  By universal property of $A[W]$, there exists a unique homomorphism $A[f]$ of $A$-algebras from $A[W]$ into $A[V]$, such that
  $A[f](y_j)=t_j(\ov x)$ for $j=1,\ldots,m$, where $y_j$, $t_j$ and $\ov x$ are identified with their interpretations.

 \begin{prop}
 The rule $\mb A:Aff_A\to \W_A$, $V\mapsto A[V]$, $(f:V\to W)\mapsto (A[f]:A[W]\to A[V])$, is a full and faithful contravariant functor.
 \end{prop}
\begin{proof}
Functoriality is clear.\\For faithfulness, suppose $f_1, f_2:V\to W$ are two morphisms and $\ov t_1$, $\ov t_2$ are two 
tuples of terms representing them by Lemma~\ref{REDMOR}. If $A\mo \forall\ov x\ [\bigwedge\pi_1(\ov x)\imp
 \ov t_1(\ov x)=\ov t_2(\ov x)]$, we have $f_1(\ov a)=f_2(\ov a)$ for all $\ov a\in V$. This means that if $f_1\neq f_2$,
 then $A\not\mo\forall\ov x\ [\bigwedge\pi_1(\ov x)\imp \ov t_1(\ov x)=\ov t_2(\ov x)]$, so 
 in particular $Th_W(A|A)\cup \pi_1\not\mo \ov t_1(\ov x)=\ov t_2(\ov x)$. As $A[V]\mo Th_W(A|A)$, we have $\ov t_1^{A[V]}(\ov x)
 \neq\ov t_2^{A[V]}(\ov x)$, so we have $A[f_1]\neq A[f_1]$ and $\mb A$ is faithful.\\
   For fullness, let $g:A[W]\to A[V]$ be a morphism of $A$-algebras : by definition of $A[V]$, for every $j=1,\ldots,m$ there
    is a term $t_j(\ov x)$ such that $g(y_j)=t_j(\ov x)$. The formula $\phi(\ov x,\ov y)=[\bigwedge\pi_1(\ov x)\wedge
    \bigwedge\pi_2(\ov y)\wedge\bigwedge_{j=1}^m y_j=t_j(\ov x)]$ defines a morphism $f$ from $V$ into $W$, such that $A[f]=g$ :
    $\mb A$ is full.
\end{proof}

We could say that all this amounts to defining affine algebraic varieties in $A$ through the coordinate algebras.
 The following section develops a formal analog to the property of ``geometric saturation'' in algebraically closed fields,
 which in some sense allows us to ``identify'' the affine algebraic geometry in $A$, with its expression in reduced $A$-algebras,
 or alternatively with the algebraic category of finitely presented such ones.
 
  \subsection{Geometrically closed structures}\label{GCSS}
For the remainder of this section, let $\t$ be a first order theory in $\ms L$ and $\K$ the elementary class of its models. The following lemma is a
pretext to introduce a relationship between quasi-algebraic sentences
  and the rational points for finite $a$-types.
\begin{lem}\label{RATYP}
If $A$ is in $\wk$ and $\pi$ is a finite $a$-type of $A[\ov x]$, then $tp_a^A(\pi(A))\supset \trap\pi$.
\end{lem}
\begin{proof}
Suppose that $\pi$ is a finite $a$-type of $A[\ov x]$ and $V=\pi(A)$. For every $\ov a\in V$, $tp_a^A(\ov a)$ contains $\pi$ so
$tp_a^A(V)=\bigcap_{\ov a\in V} tp_a^A(\ov a)$ contains $\pi$. By definition of the radical of an $a$-type, we thus have $\sqrt[\tw]{\pi}^+
\subset tp_a^A(V)$. As we have $\trap\rho=\sqrt[\tw]{\rho}^+$ for every $a$-type $\rho$, the lemma is proved.
\end{proof}

In some sense the ``geometric'' version of Hilbert's Nullstellensatz (Theorem~\ref{HILNUL}) for a field $K$ says that all polynomial equations which are
 a consequence of a finite set $\mc S$ of such polynomials, are already true in $K$ for the rational points of $S$. This led B.Plotkin
  to look at the preservation of quasi-algebraic sentences in \emph{extensions} of structures (\cite{BP}, L.3, 2), an idea
   which is already present in \cite{WW} for individual quasi-algebraic sentences. Here we introduce the same approach for algebras and
   whole sets of quasi-algebraic sentences.
\begin{defi}
\begin{itemize}
 \item A homomorphism $f:A\to B$ between two $\ms L$-structures will be called \emph{geometrically closed}, if for every
  quasi-algebraic sentence $\phi(\ov a)$ with parameters in $A$ such that $A\mo\phi(\ov a)$, then $B\mo\phi(f\ov a)$.
 \item If $\K$ is a class of $\ms L$-structures and $A$ is an object of $\K$, then we will say that $A$ is \emph{geometrically closed in $\K$}
  if every $A$-algebra in $\K$ is geometrically closed.
\end{itemize} 
\end{defi}

Notice that $f:A\to B$ being geometrically closed means exactly that $(B,f)\mo Th_W(A|A)$, if $(B,f)$ denotes the $\ms L(A)$-expansion
 of $B$ induced by $f$. In the preceding section, we worked in the class $\W(A)$, which was a ``syntactic'' equivalent to
  the subcategory of $\W_A$ consisting of geometrically closed $A$-algebras. Saying that $A$ is geometrically closed in $\wk$
   thus amounts to saying that $Th_W(A|A)$ is axiomatised by $T\cup D^+A$. This may be expressed in two ways, using either
    the radicals or the finitely presented reduced $A$-algebras.

\begin{thm}\label{CHARGC}
 Let $A$ be a model of $\tw$. $A$ is geometrically closed in $\wk$ if and only if $A$ satisfies ``Hilbert's Nullstellensatz
 relatively to $\t$'', in other words if for every finite tuple $\ov x$ of variables and every finite $a$-type
 $\pi$ of $A[\ov x]$, we have $tp^A_a(\pi(A))=\trap\pi$.
\end{thm}
\begin{proof}
 Suppose that $A$ is geometrically closed, and let $\pi$ be as in the statement of the theorem. As $A\mo\tw$, we already know
 that $tp_a^A(\pi(A))\supset\trap\pi$ by Lemma~\ref{RATYP}. A formula $\phi(\ov a, \ov x)$ (i.e., defined in $A[\ov x]$)
  is in $tp_a^A(\pi(A))$ if and only if $A\mo\forall\ov x\bigwedge \pi(\ov a,\ov b,\ov x)\imp \phi(\ov a,\ov x)$
   ($\ov b$ is the extra tuple of parameters of $A$ different from $\ov a$ and appearing in $\pi$). In this case, this last sentence $\chi$
    is quasi-algebraic and the quotients $A[\ov x]/\mf p$ are in $\wk$, because the $\mf p$'s are radical (Theorem~\ref{rep}).
     Remembering that $f_\mf p:A[\ov x]\onto A[\ov x]/\mf p$ is the quotient morphism, $f_\mf p$ is geometrically closed,
      because $A$ is, so for every $\mf p$, we have $A[\ov x]/\mf p\mo\chi^{f_\mf p}$. This exactly means that $\phi\in\trap\pi$,
       because then the equivalence class of the tuple $\ov x$ satisfies $\phi$ in every $A[\ov x]/\mf p$.\\
      Conversely, suppose that $tp^A_a(\pi(A))=\trap\pi$ for every finite $a$-type $\pi$ in finitely many variables over $A$.
      Let $f:A\to B$ be a homomorphism in $\wk$, and suppose that $\chi=\forall\ov x\bigwedge\pi(\ov a,\ov x)\imp\phi(\ov a,\ov x)$
       is a quasi-algebraic sentence with parameters in $A$ and true in $A$ : this means that $\phi\in tp_a^A(\pi(A))$, so
        by hypothesis $\phi\in\trap\pi$. Now let $\ov b$ be an $|\ov x|$-tuple of $B$ satisfying $\pi(f\ov a,\ov x)$ :
         the $a$-type $\mf p$ of all atomic formulas with parameters in $A$ and satisfied by $\ov b$ in $B$ with respect to $f$,
          is a radical $a$-type of $A[\ov x]$ containing $\pi$, because the quotient $A[\ov x]/\mf p$ is isomorphic
           to the substructure of $B$ generated by $f(A)$ and $\ov b$, and this last is in $\wk$. In short, we have $\phi\in \mf p$,
            so $B\mo\chi^f$ : $f$ is geometrically closed, as is thus $A$.
\end{proof}

In the lemma and the theorem, we only needed the \emph{radical} $a$-types, which
 means that we could have worked directly within a quasi-algebraic theory $\t$, with $\K=\wk$, without even changing the notion of
 radical $a$-types. This is strangely analogous to the minimal approach of positive model theory, where I.Ben Yaacov looks
 at existentially closed models of a $\Pi$-theory (\cite{BY}, 1.1), and which we will quickly mention in section~\ref{ST}.
 Besides, the interest of starting from an arbitrary theory $\t$ will become clear in the sequel. Let us note for the moment the
 following consequences of geometric completeness.

\begin{thm}\label{DUAL}
 If $A$ is a geometrically closed object of $\W=\wk$, then the affine coordinate algebra functor $\mb A$ introduced in~\ref{AFFQU}
 is a duality between the categories $Aff_A$ and $\W_{A}^{fin}$, the category of \emph{finitely presented reduced $A$-algebras}. 
\end{thm}
\begin{proof}
 Of course, a reduced $A$-algebra $B$ is finitely presented as such if there exists a finite set $X$ and a finite set
 of atomic sentences $P$ in $\ms L\sqcup A \sqcup X$, such that $B$ is presented in $\W$ by $(X,P\cup D^+A$). In other words,
  if $B$ is such an algebra, by Corollary~\ref{PRESWK} it is isomorphic to a quotient of the form $A[\ov x]/\trap\pi$, where $\ov x$
   is a finite tuple of variables and $\pi$ is a finite $a$-type. Let $V=\pi(A)$ be the affine algebraic variety
    defined in $A$ by $\pi$; by geometric completeness of $A$, we have $tp_a^A(\pi(A))=\trap\pi$, whereby
     $A[V]\simeq A[\ov x]/\trap\pi \simeq B$. Full and faithful $\mb A$ is essentially surjective on $\W_{A}^{fin}$,
      it is a duality. 
\end{proof}

We note a couple of things. Firstly, this section sheds some light on Hilbert's Nullstellensatz and analogous results, in that prime
 ideals have not so much to do with this property, at least from this point of view ! Indeed, algebraically closed fields are
 just geometrically closed \emph{semiprime} rings \emph{which are fields}.\\
 Secondly, if we start with a quasi-algebraic theory $\t$, what can we say about the subclass of geometrically closed models of $\t$ ?
  Is this an axiomatisable class, and if so, what may be said about such an axiomatisation ?
     Here positive model theory, or positive model completeness, comes to our rescue, but at the expense of a slight hypothesis on $\t$.
      Before we can handle this, we have to make a detour by the original Nullstellensatz, in order to characterise
       algebraically closed fields among non-trivial rings. It will appear later that the only geometrically closed
        semiprime ring which is not a field is the trivial ring, and this has a logical significance.

\section{Geometric completeness}\label{GC}

  \subsection{Some characterisations of algebraically closed fields}\label{CACF}
  G.Cherlin proved an affine analog of Hilbert's Nullstellensatz derving from the notion of model-completeness of a theory of rings
 (see section~\ref{HCNS}).
   It may be used to \emph{characterise} algebraically closed fields among \emph{integral domains} as follows. If $A$ is an integral
    domain and $X$ is a set of variables, say that a prime ideal $\mf p$ of $A[X]$ is \emph{strongly prime} if $A[X]/\mf p$ is
     an extension of $A$ : in the language of section~\ref{HCNS}, if $\t$ is the theory of algebraically closed fields and $I$ is an
     ideal of $A[X]$, then $\trad I$ is the intersection of strongly prime ideals of $A[X]$ containing $I$.
\begin{thm}\label{NSACF}
 Let $A$ be an integral domain and $\t$ be the theory of algebraically closed fields. The following are equivalent :
\begin{enumerate}
 \item $A$ is an algebraically closed field
 \item For every $n\in \N$ and every finitely generated ideal $I$ of $A[X_1,\ldots,X_n]$, one has $\ms I(\ms Z_A(I))=\sqrt I$
 \item For every $n\in \N$ and every finitely generated ideal $I$ of $A[X_1,\ldots,X_n]$, one has $\ms I(\ms Z_A(I))=\trad I$.
\end{enumerate}
\end{thm}
\begin{proof}
 (1)$\imp$(2) This is a version of Hilbert's Nullstellensatz (Theorem~\ref{HILNUL}).\\
 (2)$\imp$(3) In general, when $A$ is an integral domain, for every $x\in\ms Z(I)$ the ideal $Ker\ (e_x)$ is strongly prime and contains $I$, so
 one has $\sqrt \I \subseteq \trad \I\subseteq \bigcap_{\ms Z(I)} Ker\ (e_x)=\ms I(\ms Z(I))$. Under the hypotheses
 of (2) then, all three inclusions are equalities, which in particular entails (3).\\
(3)$\imp$(1) We first show that $A$ is a field. Let $f:A\into F$ be the embedding into the fraction field of $A$ : if $a\in A$ is not zero, there exists an inverse $b$
 of $f(a)$ in $F$. Consider the canonical embedding $i:A\into A[X]$ and the evaluation morphism $f_b$ in $b$ above $f$, so that we have
 $f_b\circ  i=f$. By choice of $b$ we get $(aX-1)\subseteq Ker\ (f_b)$ and moreover $Ker(f_b)$ is stronlgy prime, as $F$ is an integral domain and
 $Ker\ (f_b)\cap i(A)=\ (0)$. We conclude that $\trad{(aX-1)}\subseteq Ker\ (f_b)\neq (1)$, hence by hypothesis
 $\ms I(\ms Z_A(aX-1))=\trad{(aX-1)}\neq\emptyset$, which is equivalent to saying that $\ms Z_A(aX-1)\neq \emptyset$, as $A$ is a non-trivial ring. This
 is turn means that $a$ is invertible in $A$, so $A$ is a field.\\
Finally we show that $A$ is algebraically closed. Let $f:A\into K$ be an algebraic extension of $A$, and $\alpha\in K$, of minimal polynomial $P$
 over $f$. As as $(P)\cap i(A)=(0)$, the ideal $(P)$ is strongly prime, so we get $\trad{(P)}=(P)$. By the hypothesis this means that $\ms I(\ms Z_A(P))=(P)
\neq (1)$ which again means that $(P)$ has a rational point $a$ in $A$. Necessarily we have $P=X-a$, whereby $f(a)=\alpha$, so $f$ is surjective and in
 fact is an isomorphism : $A$ is algebraically closed. 
\end{proof}

Hilbert's famous theorem deals with the \emph{algebraic} radical of ideals, and not the $\t$-radical of Cherlin,
 which selects only those primes which are the ideals of rational points in \emph{extensions} of the basic domain : this is natural
  in view of the classical model-theoretic approach, which deals with \emph{extensions} of structures. If we want a true generalisation
  of the original theorem in the same spirit, we must
  suppress the restriction on the prime ideals : this may be done in the special case of rings (\cite{JB}, chapter II) and here it is formally
  achieved through the theory of prime $a$-types, presented in section~\ref{PRT}. Algebraically, this means that we have to replace
  the \emph{embeddings} of classical model theory by the \emph{homomorphisms} of positive model theory. A ``positive'' version of the
  preceding characterisation will then give us the key to connect positive model-completeness and what we will call
  ``geometric completeness''.

\begin{thm}\label{CHACF}
 For a non-trivial ring $A$, the following assertions are equivalent :
\begin{enumerate}
 \item $A$ is an algebraically closed field
 \item For every $n\in \N$, for any ideal $I$ of $A[X]=A[X_1,\ldots,X_n]$, $\ms I(\ms Z_A(I))=\sqrt I$
 \item For every $n\in \N$, for any ideal $I$ of finite type of $A[X]=A[X_1,\ldots,X_n]$, $\ms I(\ms Z_A(I))=\sqrt I$
 \item $A$ is positively existentially closed in the class of non-trivial rings.
\end{enumerate}
\end{thm}
\begin{proof}
We show the equivalence of the assertions, replacing equalities between ideals in $(1)$ and $(2)$ by direct inclusions
 ($\ms I(\ms Z_A(I))\subseteq\sqrt{I}$), and show in the course of it that this is enough to prove the theorem as it is stated. \\
$(1)\imp(2)$ This is again Hilbert's Nullstellensatz (Theorem~\ref{HILNUL}).\\
$(2)\imp(3)$ Obvious.\\
$(3)\imp(4)$ Let $I$ be a ideal of finite type over $A$, with no rational point in $A$ : this means that $\ms Z_A(I)=\emptyset$.
 It follows that $\ms I(\ms Z_A(I))=A[X]$. By hypothesis, we have $\sqrt\I\supseteq \ms I(\ms Z_A(I))$, so we get $1\in\sqrt I$,
from whence $1\in I$ and $I=A[X]$. Hence, $I$ cannot have a rational point in a non-trivial $A$-algebra. In other words,
 every non-trivial $A$-algebra is an immersion, which means that $A$ is positively existentially closed.\\
$(4)\imp(3)$ We distinguish two cases :
\begin{itemize}
 \item If $I=A[X]$, notice that $\sqrt I=A[X]$. As $A$ is not the trivial ring, $\ms Z_A(A[X])=\emptyset$, so $\ms I(\ms Z_A(A[X]))
=A[X]$. This means that $\ms I(\ms Z_A(I))=\sqrt I$.
 \item If $I\neq A[X]$, let $Q\notin \sqrt I$ : there is a prime ideal $\mf p\supset I$ such that $Q\notin\mf p$, so in the non-trivial
 algebra $A[X]/\mf p$, $I$ has a rational point out of the zero set of $Q$. In a fraction field of $A[X]/\mf p$, we find a solution
 $c$ of $I$ and an inverse $d$ for $Q(c)$, for which we can find an antecedent $c'd'$ in $A$, a solution of the zero set of $I$
 and $Q(X).Y-1$. The point $c'$ is in $\ms Z_A(I)$ but not in $\ms Z_A (Q)$, so $Q\notin \ms I(\ms Z_A(I))$, and $(\sqrt I)^c
\subseteq (\ms I(\ms Z_A(I)))^c$, hence $\ms I(\ms Z_A(I))\subseteq \sqrt I$.
\end{itemize}
(3)$\imp$(1) From the preceding characterisation of algebraically closed fields (Theorem~\ref{NSACF}), it suffices to show that $A$ is a
field. Let then $a\in A$ and let $I:=(a)$. The ``zero set'' $\ms Z_A(I)$ must be a subset of $A^0=\{\emptyset\}$ ($A^0$ is the zero-dimensional
 affine space, i.e. the set of $0$-tuples) : it is
 $A^0$ if $a=0$ and $\emptyset$ if $a\neq 0$. If then $a\neq 0$, we have $\sqrt I\supseteq \ms I(\ms Z_A(I))=A$,
 so $1\in I$ : $a$ is invertible and $A$ is a field.\\
The four assertions are therefore equivalent, if we replace the equalities in (2) and (3) by inclusions; in this case, $A$ is a field, so
 these inclusions are equalities.
\end{proof}

    \subsection{Another logical Nullstellensatz}
 In the preceding section, we established a property of ``geometric saturation'' in structures we called ``geometrically closed''.
  Here is the corresponding notion for a theory.

\begin{prop}
 A model of a theory $\t$ is geometrically closed in $\K=Mod(\t)$ if and only if it is geometrically closed in $\wk=Mod(\tw)$.
\end{prop}
\begin{proof}
 Let $A$ be a model of $\t$. Every model of $\t$ is a model of $\tw$, so geometric completeness in $\K$ is stronger
 than in $\wk$. Reciprocally, if $f:A\to B$ is a reduced $A$-algebra, then we may suppose by Lemma~\ref{REPQV} that
 $B$ is embedded in a product $\prod_{i\in I} B_i$ of models of $\t$. Consider the $i^{th}$ projection $p_i$ of the product.
 Every $p_i\circ f$ is a geometrically closed $A$-algebra by hypothesis, so their product (as an $\ms L(A)$-structure) is a model of
 $Th_W(A|A)$,
 a quasi-algebraic theory also true in the subalgebra $(B,f)$, as quasi-algebraic sentences are preserved under
 products and substructures. This means that $A$ is geometrically closed in $\wk$.
\end{proof}

\begin{defi}
 We will say that the theory $\t$ is \emph{geometrically complete} if every model of $\t$ has the property of the proposition.
\end{defi}

In general, positively existentially closed structures of an elementary class $\K=Mod(\t)$ have no reason to exhibit
a particular connection with the notion of geometrically closed objects. However, when they consist of an elementary
subclass, their homomorphisms have the strongest preservation property.
\begin{lem}\label{POSEL}
 If $\t$ is positively model complete, then every homomorphism in $\K$ is elementary.
\end{lem}
\begin{proof}
 By Proposition~\ref{PMCCOH}, every coherent formula is equivalent modulo $\t$ to the negation of another coherent formula. By an easy
 induction on the complexity of first order formulas, every formula has such a coherent complement, so it is itself by~\ref{PMCCOH} again
equivalent to a coherent formula.\\
If $f:M\to N$ is a homomorphism between models of $\t$, we then deduce that for every sentence $\phi(\ov m)$ of the elementary diagram
 of $M$, a coherent equivalent $\psi(\ov x)$ of $\phi(\ov x)$ modulo $\t$ is true in $M$ of $\ov m$, so is also true in $N$ of $f\ov m$;
 as $N$ is itself a model of $\t$, then $N\models \phi(\ov m)$ so $f$ is elementary.
\end{proof}

It should then be obvious that models of such a positively model-complete theory enjoy an abstract Nullstellensatz, which we will
state in a compact form and an explicit one.

\begin{thm}\label{POSNS}
 Every positively model-complete theory $\t$ is geometrically complete.\\
In other words, if $A$ is a model of a positively model-complete theory $\t$ and $\pi$ is a \emph{finite}
 $a$-type in finitely many variables over $A$, then $tp_a^A(\pi(A))=\trap\pi$.
\end{thm}
\begin{proof}
 Let $M$ be an object of $\K$ and let $f:M\to N$ be a homomorphism. By the preceding lemma, $f$ is elementary, so $f$ preserves
 in particular all sentences of the quasi-algebraic diagram of $M$ : $f$ is geometrically closed, so that $M$ is geometrically closed
 in $\K$.
\end{proof}

\begin{ex}
 Applying this theorem to the theory $\t$ of real-closed fields in the language of rings, we recover the real Nullstellensatz
 (\cite{BCR} 4.1.4), so the complex and real theorems are unified in this framework.\\
 Changing the point of view, we may apply the theorem to the theory of real-closed ordered fields in $\langle +,-,\times,\geq,0,1\rangle$.
  We know that if $F$ is a real closed field and $S$ is a closed semi-algebraic subset of $F^n$, then $S$ is defined
   by a finite number of polynomial (large) inequations $P(\ov x)\geq 0$  (\cite{BCR}, 2.7.2; this includes polynomial equations defined
   by $P(\ov x)\geq 0 \wedge -P(\ov x)\geq 0$). By the previous theorem and Example~\ref{EX4}, the set of such inequations which are satisfied in every point of $S$ is
   given by the intersection of all prime cones containing the polynomial $P(\ov x)$ such that the inequation $P(\ov x)\geq 0$ is in $T$. 
\end{ex}

As $tp_a^A(\pi(A))$ is a formal analog to $\ms I(\ms Z_A(\pi))$, the second statement of Theorem~\ref{POSNS} has the original algebraic
flavour, and we will carry out more concrete descriptions in the applications. For the moment, we tackle a certain ``reciprocal''
of this result.

  \subsection{Strict theories}\label{ST}
 The traditional model-complete theories of fields (with additional structure in a ``natural language'') studied in model theory
  satisfy the hypothesis of positive model-completness, like for instance real-closed ordered fields (\cite{MA}, 3.3.15), differentially
  closed fields (\cite{MA}, 4.3), or generic difference fields (\cite{CK}, 1.1). We'll look at the last two in the applications.
Positive model-completeness seems in the light of Lemma~\ref{POSEL} to be a much stronger property than geometric completeness. However, the
definition of positively existentially closed structures deals with rational points of affine algebraic varieties : a homomorphism
$f:A\to B$ is an immersion if and only if every affine variety defined in $A$ has a rational point in $A$ if it has one in $B$ through
$f$. This is a mean of connecting the two notions in the other way. In the elementary case, they are very close and even equivalent if
we make a mild additional assumption, suggested by Theorem~\ref{CHACF} and verified in every theory of field. This rests on
the observation that the positively existentially closed non-trivial rings \emph{are} the geometrically closed ones.
 
\begin{defi}\label{strict}
 The theory $\t$ will be called \emph{strict} if the one-element structure $\mb 1$ embeds in no model of $\t$; stated otherwise, if
 $\mb 1\not\models\t_\forall$ or equivalently if $\mb 1\not\mo\t_u$.
\end{defi}

The systematic connection between geometric completeness and positive model-completeness rests on the
following lemma, which is not obvious at first sight; it somehow creates a bridge in strict theories between strict universal Horn
 and h-universal formulas, which are the two kinds of (basic) universal Horn formulas.
\begin{lem}\label{GCIM}
 If $f:A\to B$ is a geometrically closed homomorphism between $\ms L$-structures and if $\mb 1\not\into B$, then $f$ is an immersion.
 In particular, if $\t$ is strict, every geometrically closed homomorphism between models of $\t_u$ is an immersion.
\end{lem}
\begin{proof}
 Suppose for the sake of contradiction that $f$ is not an immersion. This means that there exists a positive primitive sentence
 $\phi(\ov a)=\exists\ov y \bigwedge_{i=1}^n \psi_i(\ov a,\ov y)$ with parameters in $A$, such that $A\not\models\phi(\ov a)$ but $(B,f)\models \phi(\ov a)$. There is then
 no rational point for $\bigwedge_{i=1}^n \psi_i(\ov a,\ov y)$ in $A$, so for every atomic formula $\chi(\ov a,\ov a',\ov y)$
  with parameters in $A$, the quasi-algebraic sentence $\theta(\ov a):=\forall\ov y\ [\bigwedge_{i=1}^n \psi_i(\ov a,\ov y)\imp
  \chi(\ov a,\ov a',\ov y)]$ is true in $A$.
 Let $\ov b$ be a rational point of $\bigwedge_{i=1}^n \psi_i(f\ov a,\ov y)$ in $B$. As $f$ is geometrically closed, we have
 $(B,f)\models \theta(f\ov a)$, so $\ov b$ satifies every atomic formula $\chi(f\ov a,f\ov a',\ov y)$ in $B$. In particular, this means that all
 the coordinates of $\ov b$ are equal to a single element $b$, and that this element generates in $B$ the one-element structure, which
 is impossible by hypothesis. We conclude by contradiction that the hypothesis is false : $f$ is an immersion. 
\end{proof}

\begin{rem}
 This also means that if $\t$ is a strict theory, $\K=Mod(\t)$ and $A$ is a geometrically closed non-trivial structure in $\wk$,
 then $A$ is positively existentially closed as a model of $\t_u$.
\end{rem}
 Indeed, as $A$ is not $\mb 1$, then $A$ embeds into a non-trivial product of models of $\t$, which means that $A\mo\t_u$. For those
 ``non-trivial geometrically closed models of $\tw$'', we more clearly enter into the realm of positive model theory of \cite{BY} and \cite{BYP}. 

It is not clear to us why the reciprocal of the lemma should be true in general, i.e. that any non-trivial model of $\tw$
 positively existentially closed as a model of $\t_u$ would be geometrically closed in a strict theory, but we have no counter-example.
 In the elementary case, however, we have the reciprocal of Theorem~\ref{POSNS}.

\begin{thm}\label{PMCGC}
 If $\t$ is a strict theory, then it is positively model-complete if and only if it is geometrically complete.
\end{thm}
\begin{proof}
 The direct part of the equivalence is Theorem~\ref{POSNS}, while for the other part, if every objet of $\K=Mod(\t)$ is geometrically
 closed in $\K$, then every homomorphism in $\K$ is an immersion by Lemma~\ref{GCIM}, i.e. every objet of $\K$ is positively existentially
 closed. As $\K$ is elementary, this exactly says that $\t$ is positively model-complete.
\end{proof}

   \subsection{Atomic Morleyisation and existential completeness}\label{AMC}
We remind that ``Cherlin's Nullstellensatz'' builds on model-completeness, rather than on positive model-completeness
 (section~\ref{HCNS}). A certain way of turning theories into positive model-complete ones is the operation
  called ``Positive Morleyisation'' by Ben Yaacov and Poizat (\cite{BY}, \cite{BYP}). Here we introduce a very tame version of it,
   adding relation symbols for the negations of atomic formulas only (we don't ``Morleyise'' all the fragment generated by them), which
   is somehow handled in \cite{BS} : this will allow us to connect our work with an extension of Cherlin's.\\
If $\ms L$ is a first order language, we associate to $\ms L$ another language $\ms L^*$, consisting of $\ms L$ and the addition,
 for every relation symbol $R$ of $\ms L$, of a new relation symbol $R^*$ of the same arity (or sorting), which will axiomatise
  the complement of $R$ in $\ms L^*$-structures. 
  \begin{defi}
   If $\t$ is a theory in $\ms L$, then say that the \emph{atomic Morleyisation of $\t$} is the theory $\t^*$ in $\ms L^*$,
    which consists of $\t$ and all the axioms of the form $\forall\ov x\ [\neg R(\ov x)\iff R^*(\ov x)]$,
 where $R$ is a relation symbol of $\ms L$ and $\ov x$ an appropriate tuple of variables.
  \end{defi}
If $A$ is an $\ms L$-structure, then we will denote by $A^*$ the $\ms L^*$-structure which is the expansion of $A$ by the
 symbols $R^*$, interpreted in $A$ as the complements of $R$. As any $\ms L$-structure is a model of the empty theory $\emptyset$,
 this process produces for every $A$ a model $A^*$ of $\emptyset^*$. In fact, it turns every model $A$ of a theory $\t$
 into a model of $\t^*$. Note that the additionnal axioms could be replaced by the couples of basic universal ones
 $\forall\ov x\ [R(\ov x)\wedge R^*(\ov x)\imp \bot]$ and $\forall\ov x\ [\top\imp R(\ov x)\vee R^*(\ov x)]$. One way of rephrasing
 (classical) model-completeness is the following.
 \begin{rem}
  A theory $\t$ in $\ms L$ is model-complete if and only if $\t^*$ is positively model-complete.
 \end{rem}
Let now $\ms L=\langle +,-,\times,0,1 \rangle$ be the language of ``rings'', $\t$ a theory of rings in this language,
 $A$ a model of $\tu$ and $I$ an ideal of a polynomial ring $A[\ov X]=A[X_1,\ldots,X_n]$, as in section~\ref{HCNS}. An ideal $\mf p$
 partakes in the $\t$-radical $\sqrt[\t]{I}$ of $I$ if it contains $I$ and if $A[\ov X]/\mf p$ is an \emph{extension} of $A$ which
   embeds into a model of $\t$. Now through atomic Morleyisation, a map $f:A\to B$ of $\ms L$-structures is an embedding
   \emph{if and only if} it is a homomorphism of $\ms L^*$-structures from $A^*$
 into $B^*$. In other words, the ideals $\mf p$ of the kind evoked ar exactly the algebraic kernels of the ring homomorphisms
     induced by $\ms L^*$-homomorphisms $f:A^*[\ov X]\to B^*$, where $B$ is a model of $\t$ and $A^*[\ov X]$ is the expansion of $A[\ov X]$
      by the $\ms L^*$-structure \emph{on $A$}. This means that they correspond
     to $\t^*$-prime $a$-types of $A[\ov X]$, containing the set $\pi$ of equations induced by $I$ : the positive $\t^*$-radical
      $\sqrt[\t^*]\pi^+$ corresponds to $\trad I$, so we connect here with \ref{CH}. Let us do it systematically for any
 language $\ms L$.

\begin{defi}
 If $\t$ is a first order theory, $A$ is an $\ms L$-structure, $\ov x$ is a finite tuple of variables and $\pi$ is an $a$-type of
 $A[\ov x]$, say that
\begin{itemize}
 \item $\pi$ is \emph{strongly ($\t$-)prime}, if it is prime and if the natural composite morphism $A\into A[\ov x]\onto A[\ov x]/\pi$
 is an embedding
 \item the \emph{strong radical} of $\pi$, noted $\trad\pi$, is the intersection of all strong primes containing $\pi$
 \item $\pi$ is \emph{(strongly $\t$-)radical}, if $\pi=\trad\pi$.
\end{itemize}
\end{defi}

The terminology ``strong'' comes from the fact that every strongly prime $a$-type is prime and so $\pi\subset\trap\pi\subset\trad\pi$
 for every $a$-type $\pi$, thus every strongly radical $a$-type is (positively) radical. The next ``affine'' Nullstellensatz is
 a ge\-ne\-ra\-li\-sa\-tion of Theorem~\ref{CH}; note that strong primality and radicality are defined only for $a$-types in term algebras. In the
  following analogue of Theorem~\ref{POSNS}, we don't have to assume the axiomatisability of the class of existentially closed models.
\begin{thm}\label{EXC}
 If $\t$ is an $\ms L$-theory and $A$ is an existentially closed model of $\t$, then for every finite $a$-type $\pi$ of $A[\ov x]$, we have
 $tp_a^A(\pi(A))=\trad\pi$.
\end{thm}
\begin{proof}
 Let $V=\pi(A)$. For every point $\ov a$ of $V$, we have the evaluation morphism of $A$-algebras $e_{\ov a}:A[\ov x]\onto A$, the $a$-type
  of which is $tp_a^A(\ov a)$ and $A[\ov x]/tp_a^A(\ov a)\simeq A$ as $A$-algebras, so $tp_a^A(\ov a)$ is strongly prime because
 $A\mo\tu$. This means that $tp_a^A(V)\supset\trad\pi$.\\
Reciprocally, if $\phi(\ov x,\ov a)\notin\trad\pi$, then there exists a a strongly prime $a$-type $\mf p$ of $A[\ov x]$, containing
 $\pi$ and such that $\phi\notin\mf p$. In other words, we have $A[\ov x]/\mf p\mo\bigwedge\pi(\ov x)\wedge\neg\phi(\ov x)$, where
 $\ov x$ now denotes in the formula its canonical interpretation in $A[\ov x]$. By hypothesis on $\mf p$, the algebra $f:A\to A[\ov x]/\mf p$ is an
 embedding of $A$ into a model of $\tu$. As $A$ is existentially closed for models of $\t$, it is as well for models of $\tu$,
 so $f$ is existentially closed and $A\mo\exists\ov x\ \bigwedge\pi(\ov x)\wedge\neg\phi(\ov x)$ : this means that
 $\phi\notin\bigcap_{\ov a\in V} tp_a^A(\ov a)=tp_a^A(V)$, so $tp_a^A(V)\subset\trad\pi$.
\end{proof}
The notions of this theorem may be conceptualised in the preceding framework.
\begin{defi}
 $A$ is \emph{weakly geometrically closed} in a class $\K$ if every \emph{extension} $f:A\into B$ in $\K$ is geometrically
 closed.
\end{defi}

\begin{lem}
 If $\K=Mod(\t)$ and $A\in \uk$, then $A$ is weakly geometrically closed in $\uk$ if and only if for every affine variety $V=\pi(A)$ over $A$,
 we have $tp_a^A(V)=\trad\pi$.
\end{lem}
\begin{proof}
 Suppose that $A$ is weakly geometrically closed. As in the first part of the theorem, we have $tp_a^A(V)\supset\trad\pi$. If $\phi\in tp_a^A(V)$, then
 $A\mo\forall\ov x\ [\bigwedge\pi(\ov x)\imp \phi(\ov x)]$, this last sentence $\chi$ being in $Th_W(A|A)$, as $\pi$ is finite.
 By weak geometric completeness, for every strongly prime $a$-type $\mf p\supset\pi$, the extension $f:A\into A[\ov x]/\mf p=B$
 is geometrically closed, so $(B,f)\mo\chi$, which means that $\phi\in\mf p$, and $tp_a^A(V)\subset\trad\pi$.\\
Reciprocally, suppose that $tp_a^A(V)=\trad\pi$ for every affine $V$ and that $f:A\into B$ is an extension in $\uk$. We want to show
 that $(B,f)\mo Th_W(A|A)$. Let $\chi=\forall\ov x\ [\bigwedge\pi(\ov x)\imp \phi(\ov x)]$ be one of its elements and $V=\pi(A)$; by definition
  of $Th_W(A|A)$ and $tp_a^A(V)$, we have $\phi\in tp_a^A(V)$. If $\ov b$ is in $\pi(B)$, the $a$-type of the evaluation morphism of
 $A$-algebras $f_{\ov b}:A[\ov x]\to B$ is a strongly prime $a$-type containing $\pi$. By hypothesis $\phi\in\trad\pi$, which
 means that $(B,f)\mo\phi(\ov b)$, thus $(B,f)\mo\chi$ : $f$ is geometrically closed.
\end{proof}

Another way of stating Theorem~\ref{EXC} is thus to say that an existentially closed model of any theory $\t$ is weakly geometrically
 closed for $\tu$; Proposition~\ref{EXM} will state something stronger. Notice first that if $A\subset B$ is an extension of $\ms L$-structures, it induces an extension $A^*\subset B^*$ of the corresponding
 $\ms L^*$-expansions.
\begin{lem}\label{USTAR}
 The theories $(\t^*)_\forall$ and $(\tu)^*$ are equivalent. We will note this theory $\tu^*$ and $\uk^*$ the class of
 its models.
\end{lem}
\begin{proof}
 As for any theory in $\ms L$, $(\tu)^*=\tu\cup\emptyset^*$. If $A\mo(\t^*)_\forall$, there is an embedding $f:A\into B$
 of $\ms L^*$-structures, such that $B\mo\t^*$. This means that $A\mo\tu$, and as $A\mo\emptyset^*$ (which is universal
  in $\ms L^*$), we have $A\mo(\tu)^*$.\\
Reciprocally, if $A\mo(\tu)^*$, then $A\mo\tu$, so there is an embedding of $\ms L$-structures $f:A_0\into B$,
 where $A_0$ is the $\ms L$-reduct of $A$ and $B$ is a model of $\t$. This means that $B^*\mo\t^*$, and as $(A_0)^*=A$, $f$ induces an
 embedding of $\ms L^*$-structures $A\into B^*$ and we get $A\mo(\t^*)_\forall$.
\end{proof}
Remark that in general, if $A$ is an $\ms L^*$-structure, there is no reason why we should have $(A_0)^*=A$. This is the case if and only
 if $A\mo\emptyset^*$. Call such an $\ms L^*$-structure \emph{regular}.

 \begin{prop}\label{EXM}
  If $A$ is an existentially closed model of a theory $\t$, then $A^*$ is geometrically closed in $\uk^*=Mod(\tu^*)$.
 \end{prop}
 \begin{proof}
 Let $f:A^*\to B$ be an $\ms L^*$-homomorphism in $\uk^*$ : $f$ is an embedding because $A^*,B\mo\emptyset^*$. Let $\chi=\forall\ov x\ \bigwedge\Phi(\ov x,\ov a)\imp\psi(\ov x,\ov a)$
  be a quasi-algebraic sentence of the diagram $Th_W(A^*|A^*)$. In every regular $\ms L^*$-extension of $A^*$, this sentence is equivalent
 (as any universal $\ms L^*(A)$-sentence) to a universal $\ms L(A)$ sentence $\chi'$. As $A$ is existentially closed in $\uk$,
 $f:A\into B_0$ is existentially closed, so $(B_0,f)\mo\chi'$ and then $(B,f)\mo\chi$, because $B$ is regular. The morphism $f$ is thus
 $\ms L^*$-geometrically closed, and so is $A^*$ in $\uk^*$.
    \end{proof}
    
The reciprocal of this proposition will lead us to another understanding of model-completeness and geometric completeness : in some sense,
 this last notion is a generalisation of model-completeness, parallel and strongly connected to positive model-completeness. This comes
  naturally because atomic Morleyisation produces solely strict theories.
 
\begin{lem}\label{STEX}
 The atomic Morleyisation $\t^*$ of $\t$ is a strict theory in $\ms L^*$.
\end{lem}
\begin{proof}
 Let $R$ be a relation symbol of $\ms L$ of arity $n$. If $A\mo\t^*$, then $R^{*A}=A^n-R^A$. If the trivial structure $\mb 1$ in $\ms L^*$
  embeds in $A$, if we call $0$ its only element, then we must have $A\mo R(0,\ldots,0)\wedge R^*(0,\ldots,0)$, because in $\mb 1$ all
   interpretations of relation symbols are full. This contradicts the regularity of $A$, so $\mb 1$ does not embed in $A$ and $\t^*$ is strict.
\end{proof}

\begin{thm}
 Let $\t$ be a theory and $A$ an object of $\uk=Mod(\tu)$. $A$ is existentially closed in $\uk$ if and only if
 $A^*$ is geometrically closed in $\uk^*=Mod(\t^*)$.
\end{thm}
\begin{proof}
 The direct part is a restatement of Proposition~\ref{EXM}, so we only need to prove the converse. Suppose that $A^*$ is geometrically closed in $\uk^*$.
  If $f:A\into B$ is an extension in $\uk$, $\mb 1$ does not embed in $B^*$ by Lemmas~\ref{USTAR} and~\ref{STEX}. As $f:A^*\to B^*$ is
  geometrically closed, by Lemma~\ref{GCIM} it is an $\ms L^*$-immersion, i.e. an $\ms L$-existentially closed embedding.  
\end{proof}

As $\tu^*$ is stronger than $(\t^*)_u$, the property of the theorem is also equivalent to ``$A^*$ is a positively
existentially closed model of $(\t^*)_u$''. As a corollary, we sum up the following characterisations of model-complete
theories (note that the preceding theorem is stronger).

\begin{cor}\label{CARMC}
For any first-order theory $\t$, the following assertions are equivalent :
 \begin{enumerate}
  \item $\t$ is model-complete
  \item $\t^*$ is positively model-complete
  \item $\t^*$ is geometrically complete.
 \end{enumerate}
\end{cor}

This raises the questions of the relations between atomic Morleyisation and prime and strongly prime $a$-types, and in what
 measure it is possible to identify geometric completeness and weak geometric completeness. For this purpose, let $\t$ be any theory in $\ms L$ and $A$ an $\ms L$-structure. Suppose that $\mf p$ is a strongly prime $a$-type of $A[\ov x]$
 and consider the homomorphisms $j:A\to A[\ov x]$ and $f:A[\ov x]\to A[\ov x]/\mf p$, the composition of which we call $g$.
 As before,
  if $A^*$ denotes the canonical expansion of $A$, then $A^*[\ov x]$ is an $\ms L^*$-expansion of $A[\ov x]$, though
   not the canonical one. However, it should be obvious that $j:A^*\to A^*[\ov x]$ is still an $\ms L^*$-homomorphism (in fact, even an embedding).
    Call $B$ the structure $A[\ov x]/\mf p$ : the map $f$ defines an $\ms L^*$-homomorphism $A^*[\ov x]\to B^*$. Indeed,
    $f$ is already an $\ms L$-homomorphism and if $R$ is a relational $n$-ary symbol in $\ms L$ and $t_1,\ldots,t_n$
     are $n$ elements of $A^*[\ov x]$, i.e. terms with parameters in $A\sqcup \ov x$ such that $A^*[\ov x]\mo R^*(t_1,\ldots,t_n)$,
      then there must be $\ov a=(a_1,\ldots,a_n)$ in $A$ such that $t_i=j(a_i)$ for every $i$, because the relational interpretation of
      $R^*$ in $A^*[\ov x]$ is empty ``outside $j(A)$''. Then we have $A^*\mo R^*(\ov a)$, so $A\not\mo R(\ov a)$. As $g$ is
      an $\ms L$-embedding (because $\mf p$ is strongly prime), we have $(B,g)\not\mo R(\ov a)$, i.e. $(B^*,f)\mo R^*(\ov t)$.\\
We are now interested in the $a$-type $\mf P$ of $f$ in $\ms L^*$ : certainly we have $\mf p\subset \mf P$.
Suppose that $\phi=R(t_1,\ldots,t_n)$ is an atomic $\ms L(A\sqcup \ov x)$-sentence not in $\mf p$. This means that
$(B,f)\not\mo\phi$, so $(B^*,f)\mo\phi^*$ (if $\phi^*=R^*(t_1,\ldots,t_n)$) and $\phi^*\in \mf P$. Reciprocally,
 if $\phi$ is an atomic $\ms L^*(A\sqcup \ov x)$-sentence in $\mf P$, either $\phi$ is atomic for $\ms L$ and then as
  $(B^*,f)\mo\phi$, we have $(B,f)\mo\phi$, so $\phi\in \mf p$; or $\phi=\psi^*$ with $\psi$ atomic for $\ms L$ and
   then $(B^*,f)\not\mo\psi$ ($\psi^*$ is the obvious equivalent of $\neg\phi$ in $\ms L^*(A)$ modulo $\emptyset^*$).
In short, if we call $\mf p^*$ the set of formulas $\phi*$ such that $\phi$ is atomic in $\ms L(A\sqcup \ov x)$ and $\phi\notin\mf p$, then
 we have $\mf P=\mf p\cup \mf p^*$.\\
The same argument shows that if $\mf P$ is a prime $a$-type of $A^*[\ov x]$ in the language $\ms L^*(A\sqcup \ov x)$,
 $B$ is the $\ms L$-reduct of $A^*[\ov x]/\mf P$ and $\mf p$ is the $a$-type of $A[\ov x]\to B$ in $\ms L(A\sqcup \ov x)$, then we have
 $\mf P=\mf p\cup \mf p^*$ with the same conventions. This in particular means that $(A[\ov x]/\mf p)^*=A^*[\ov x]/\mf P$,
  because the two members are models of $\tu^*$ whose $\ms L$-reducts are equal. This will be completed in
  the next proposition.

\begin{lem}\label{CANEXP}
 If $A$ and $B$ are $\ms L$-structures and $f:A^*\to B^*$ is an $\ms L^*$-ho\-mo\-mor\-phism of their canonical $\ms L^*$-expansions,
  then $f$ is an embedding.
\end{lem}
\begin{proof}
 If $R$ is an $n$-ary relation symbol of $\ms L$ and $B^*\mo R(f\ov a)$, then $B^*\not\mo R^*(f\ov a)$, so $A^*\not\mo R^*(\ov a)$
  because $f$ is a homomorphism. This in turn means that $A^*\mo R(\ov a)$, so $f$ is an embedding.
\end{proof}
 
\begin{prop}
If $\pi$ is an $a$-type of $A[\ov x]$, then the operation $\mf p\mapsto \mf P$ defines a bijection between strongly prime $a$-types
 of $A[\ov x]$ containing $\pi$ in $\ms L$ and prime $a$-types (in $\ms L^*$) of $A^*[\ov x]$ containing $\pi$.
\end{prop}
\begin{proof}
 With the same notations, let $B=A[\ov x]/\mf p$. We have $B\mo\tu$, so $B^*\mo\tu^*$, and $\mf P$ is the $a$-type in $\ms L^*$ of the
 $\ms L^*$-homomorphism $A^*[\ov x]\to B^*$. This means that $\mf P$ is prime and an $a$-type $\pi$ of $A[\ov x]$ is contained in
 $\mf p$ if and only if $\pi\subset\mf P$.\\
If the two strongly prime $a$-types $\mf p$ and $\mf q$ are different, then obviously we get two different primes $\mf P$ and
$\mf Q$ and $\mf p\mapsto \mf P$ is injective. If $\mf P$ is a prime $a$-type of $A^*[\ov x]$, let $\mf p$ be the $a$-type in
$\ms L$ of the homomorphism $A^*[\ov x]\onto A^*[\ov x]/\mf P$. We have $B^*=(A[\ov x]/\mf p)^*=A^*[\ov x]/\mf P$ by the preceding
discussion ($tp_{a,\ms L^*}(f)=\mf P$) so clearly we have $\mf p\cup \mf p^*=\mf P$.
  Again, as $A^*[\ov x]/\mf P$ is the canonical expansion of $A[\ov x]/\mf p$, the composite morphism
  $A^*\to A^*[\ov x] \to A^*[\ov x]/\mf P$ is an $\ms L^*$-embedding by the preceding lemma, whereby
  $\mf p$ is strongly prime. Finally, the map $\mf p\mapsto \mf P$ is a bijection.
\end{proof}

\begin{thm}
For an $\ms L$-theory $\t$, let $\W_{\K^*}= Mod((\t^*)_W)$ be the quasivariety generated in $\ms L^*$ by $K^*=Mod(\t^*)$. A structure $A$ is in $\uk$ if and only if $A^*$ is in $\W_{\K^*}$.
\end{thm}
\begin{proof}
First, if $A$ is in $\uk$, then $A^*$ is in $\uk^*$, which is a subclass of $\W_{\K^*}$.\\
Reciprocally, let $A$ be an $\ms L$-structure such that $A^*$ is a model of $(\t^*)_W$. By Lemma~\ref{REPQV}, there exists a family
 $(B_i)_{i\in I}$ of models of $\t^*$ and an $\ms L^*$-embedding $f:A^*\into \prod_{i\in I} B_i$ : call $p_i$ the $i^{th}$ projection of
 the product. If $I$ is non-empty and $i\in I$, then $p_i\circ f$ is an $\ms L^*$-homomorphism between regular $\ms L^*$-structures
  (as $B_i\mo\emptyset^*$), so by Lemma~\ref{CANEXP} it is an $\ms L^*$-embedding. In other words, $A^*$ embeds in $B_i$, so it is in
  $\uk^*$, which in turn means that $A\in\uk$. If now $I=\emptyset$, then $A^*$ embeds in the trivial structure $\mb 1$ for $\ms L^*$.
  This is impossible, because $A^*$ is regular.
\end{proof}
Remark that this does not mean that the quasivariety $\wk^*$ generated by $\K^*=Mod(\t^*)$ is $\uk^*$, because $\mb 1$ is in $\W_{\K^*}$,
 whereas it is certainly not in $\uk^*$. In fact, the theorem shows that $\W_{\K^*}$ is essentially $\uk^*\cup\{\mb 1\}$.
 
\section{Some examples}\label{APP}

  \subsection{Theories of fields and affine algebraic geometry}
The archetypical example falling under Theorem~\ref{PMCGC} is the case of theories of pure fields in the
language of rings. If $\t$ is such a theory, then as every homomorphism of fields is an embedding, model-completeness
  and positive model-completeness are equivalent. But as no field contains the trivial ring, then the geometric completeness
   of $Mod(\t)$ is also equivalent to its positive model-completeness. Hence, in theories of pure fields (or even fields
    with an additional operational structure), the three notions coïncide (beware that this is not Corollary~\ref{CARMC}).\\
    This has several implications. First, if we know that an analog of Hilbert's Nullstellensatz holds for the models
     of such a theory $\t$, then we know that $\t$ is (positively) model-complete, a property which can be nice to work
      with in model theory. Second, if we know that $\t$ is model-complete, then this gives us automatically
      a relative Nullstellensatz for the models of $\t$, as well as a first theory of affine algebraic invariants,
       which by noetherianity of the models are finitely presented ``relatively'' reduced algebras. This joins
        in spirit the project of B.Plotkin of a ``universal algebraic geometry'', in somewhat a different perspective though
        (\cite{BP}). As model-complete theories of fields arise naturally in arithmetic in a continuous number
         (see \cite{MK}), it could be the subject of a treatment of its own, searching for generalisations of common
          results in complex and real affine algebraic geometry. This would however exceed the bounds of the present paper.

  \subsection{Group-based algebras}\label{GBA}
In section~\ref{VAR}, we introduced varieties of group-based algebras. It is possible to apply the results on geometric
 completeness to this situation and then generalise Hilbert's Nullstellensatz and the affine duality in its original
 ``geometric'' form. Let then $\ms L$ be an equational language and $\V=Mod(\t_0)$ be a variety of group-based algebras in $\ms L$. Let $\t$ be a theory of group-based
  algebras in $\ms L$, stronger than $\t_0$, i.e. such that $Mod(\t)=\K\subset \V$. The theory of $a$-types creates the bridge with
 the algebraic language.
\begin{lem}
 If $A$ is an algebra of $\V$, there is a bijection between closed $a$-types and ideals of $A$.
\end{lem}
\begin{proof}
 If $\pi$ is a closed $a$-type of $A$, then $A/\pi$ is a homomorphic image of $A$, so lies in $\V$. Define
 $\Phi(\pi)$ as the algebraic kernel of the canonical projection : this is an ideal of $A$. If $\pi\neq\pi'$ are two such
 closed $a$-types, then as $\ms L$ is equational there is an equation $a=b$
 which is in $\pi-\pi'$, for example. This means that $a*b^{-1}\in\Phi(\pi)-\Phi(\pi')$, so $\Phi$ is injective. If
 now $I$ is an ideal of $A$, let $\pi=tp_a(A\onto A/I)$. For every couple $(a,b)$ of elements of $A$, the equation $a=b$
 is in $\pi$ if and only if $a*b^{-1}$ is in $I$, so clearly $\pi$ is closed and $\Phi(\pi)=I$.
\end{proof}

This means that changing from $a$-types to ideals does not raise any problem for applying what precedes.
 In particular, we may define prime and radical ideals, as in Definition~\ref{PRAT}.
\begin{defi}
 Let $A$ be an algebra in $\V$ and $I$ be an ideal of $A$.
\begin{itemize}
 \item $I$ is \emph{prime} if the correspondong $a$-type is, i.e. if $A/I\mo\tu$
 \item $I$ is \emph{radical} if the corresponding $a$-type is, i.e. if $A/I\mo\tw$ (by Theorem~\ref{REP})
 \item the radical of $I$, noted $\trap I$, is the intersection of all prime ideals $\mf p$ of $A$ such that $I\subset \mf p$.
\end{itemize}
\end{defi}

The representation theorem (\ref{rep}) may then be stated in this language, and the previous results may be applied.
With the same notations as in algebraic geometry (cf section~\ref{CACF}), if $A$ is an algebra in $\V$,
 $X_1,\ldots,X_n$ are indeterminates and $I$ is an ideal of $A[X_1,\ldots,X_n]$, put $\ms Z_A(I)=\{\ov a\in A^n:\forall P
\in A[\ov X],\ P(\ov a)=e\}$ and if $S\subset A^n$, put $\ms I(S)=\{P\in A[\ov X]:\forall \ov a\in S,\ P(\ov a)=e\}$.

\begin{thm}
 If $\t$ is positively model-complete, then for every model $A$ of $\t$, for every natural number $n$ and every
 finitely generated ideal $I$ of $A[X_1,\ldots,X_n]$, one has $\ms I(\ms Z_A(I))=\trap I$.
\end{thm}

This means that we have ``the same'' Nullstellensatz in the following situations.
\begin{ex}\label{APEX1}
 If $\V$ is the class of differential rings and $\K$ the class of differentially closed fields of characteristic $0$, in 
the language $\langle +,-,\times,d,0,1\rangle$ of differential fields, then $\V$ is a variety of group-based algebras,
 and the ``ideals'' of an algebra $A$ are its differential ideals. Prime ideals are the prime differential ideals containing
 no nonzero integer and radical ideals are the radical differential ideals. These are proper if they contain no nonzero integer.\\
We also mention the analogous case of the variety $\V$ of $D$-rings and the class $\K$ of Hasse closed fields with fixed positive
 characteristic (see~\cite{FB}).
\end{ex}
\begin{ex}\label{APEX2}
 If $\V$ is the variety of rings with an additional endomorphism $\sigma$ in the language $\langle +,-,\times,\sigma,0,1\rangle$,
 $\V$ consists of group-based algebras, for which the ideals of an algebra are the ideals closed under $\sigma$ (``$\sigma$-closed''
 in \cite{CK}). They are radical if and only if they are perfect as $\sigma$-ideals (extend the definition from difference
 rings to rings with an endomorphism), and they are prime if and only if they are as $\sigma$-ideals.
\end{ex}

In these examples the theorem of Ritt-Raudenbusch in differential fields (\cite{MA2}) and its analog in difference
 fields (\cite{CK}, 3.8 and \cite{RR}, 10), establish a noetherianity condition on radical ideals (in the sense of
 group-based algebras) in a finite number of variables on the models of the theory $\t$. It is possible
 to systematise this and go in the further direction of defining in general a ``Zariski topology'' in this situation.

\begin{defi}
Suppose that $\V=\V_\K$, the smallest variety containing $\K=Mod(\t)$.
\begin{itemize}
 \item An algebra $A$ of $\V$ is \emph{radically noetherian} if for every finite tuple $X_1,\ldots,X_n$ of
 indeterminates and every \emph{radical} ideal $I$ of $A[X_1,\ldots,X_n]$, $I$ is finitely generated as an ideal
 \item the theory $\t$ is \emph{radically notherian} if every model $A$ of $\t$ is radically notherian.
\end{itemize}
\end{defi}

\begin{rem}
 If $\K=\V_\K$ is already a variety, then radical noetherianity amounts to ``classical noetherianity''. For exemple, if
 $\ms L=\langle +,-,\times,0,1\rangle$ and $\t$ is the theory of commutative rings in $\ms L$, an ideal of a ring $A$
  is radically noetherian in this sense if and only if it is noetherian in the algebraic sense. The definition then
   takes care of this situation as well as of the more restrictive ones aforementioned.
\end{rem}
 For the remainder of this section, let $\V$ denote $\V_\K$; note that $\K\subset \wk\subset \V$, because
  $\V$ is a quasivariety. In section~\ref{AFFQU}, we established the duality between affine algebraic varieties in an
  object of $\wk$ and their coordinate algebras. 
In varieties of group-based algebras (the reader should be careful not to merge affine algebraic varieties with varieties of
 algebras), the coordinate algebras may be described as in ring theory : if $A$ is an object of $\V$, $V$ is an affine
   variety in $A$ defined by the ``polynomials'' $f_1(\ov X),\ldots,f_m(\ov X)\in A[X_1,\ldots,X_n]$, and if $I$ is the
   ideal generated in $A[X_1,\ldots,X_n]$ by the $f_j$'s, then $A[V]=A[\ov X]/\ms I(\ms Z_A(I))$.
   
\begin{thm}
 If $A$ is a radically noetherian and geometrically closed model of a theory $\t$ of group-based algebras in the equational language
 $\ms L$, then the coordinate algebra functor $\mb A$ is a duality between the category $Aff_A$ of affines algebraic
 varieties and the category of finitely \emph{generated} reduced $A$-algebras.
\end{thm}
\begin{proof}
By theorem~\ref{DUAL}, as $A$ is geometrically closed in $\W=\wk$, all we need to show is that every reduced $A$-algebra of finite type is
finitely presented \emph{in $\W_A$}. Let $B\simeq A[X_1,\ldots,X_n]/I$ be such an algebra. As $B$ is reduced, one also has $B\in \V$, so
as $A$ is radically noetherian, the ideal $I$, which is radical, is generated by a finite number $f_1,\ldots,f_m$ of polynomials
 of $A[\ov X]$. But this means that $B$ is finitely presented as an $A$-algebra in $\W$ (and in fact, in $\V$).
\end{proof}
All this may be carried out in a greater generality, and in the context of infinitary first order logic. 
 For instance, the notion of ``ideal'' may be defined in every (infinitary) quasivariety from the notion
 of $a$-type, using again quasi-algebraic sentences (see \cite{JB}, chapter 6, section 5).

\end{document}